\documentclass[11pt]{amsart}
\usepackage{amssymb,amsfonts,amsmath, enumerate, a4wide, psfrag, epsfig}
\usepackage[english, frenchb]{babel}

\newcommand{\C}{\mathbf{C}}
\newcommand{\N}{\mathbf{N}}
\newcommand{\la}{\lambda}
\newcommand{\bfP}{\mathbf{P}}

\newcommand{\cc}{\mathbf{C}}

\newcommand{\re}{\mathbf{R}}
\newcommand{\dd}{\mathbf{D}}
\newcommand{\zz}{\mathbf{Z}}
\newcommand{\nn}{\mathbf{N}}
\newcommand{\cv}{\rightarrow}
\newcommand{\fr}{\partial}

\newcommand{\set}[1]{\left\{#1\right\}}
\newcommand{\norm}[1]{\left\Vert#1\right\Vert}
\newcommand{\abs}[1]{\left\vert#1\right\vert}
\newcommand{\hatcheck}[1]{\hspace{.2em}\widehat{\hspace{-.2em}\check#1}}
\newcommand{\cd}{\cc^2}
\newcommand{\pd}{{\mathbf{P}^2}}
\newcommand{\pu}{{\mathbf{P}^1}}
\newcommand{\rest}[1]{ \arrowvert_{#1}}
\newcommand{\m}{{\bf M}}
\newcommand{\unsur}[1]{\frac{1}{#1}}
\newcommand{\el}{\mathcal{L}}
\newcommand{\qq}{\mathcal{Q}}

\newcommand{\itm}{\item[-]}
\newcommand{\geom}{\dot{\wedge}\hspace{.15em}}
\newcommand{\xcd}{\hat{X}^{\mathrm{dist}} }
\newcommand{\e}{\varepsilon}
 
\DeclareMathOperator{\area}{Area}
\DeclareMathOperator{\supp}{Supp}

\newtheorem{lem}{Lemma}[section]
\newtheorem{pro}[lem]{Proposition}
\newtheorem{defi}[lem]{Definition}

\newtheorem{thm}[lem]{Theorem}
\newtheorem{cor}[lem]{Corollary}
\newtheorem{theo}{Theorem}
\newtheorem{coro}[theo]{Corollary}

\theoremstyle{remark}
\newtheorem{rqe}[lem]{Remark}

\newtheorem{exa}[lem]{Example}

\date{\today}

\title[Geometric currents and ergodic theory]{Dynamics of meromorphic maps with
small topological degree III: geometric currents and ergodic theory}

\author{Jeffrey Diller \and Romain Dujardin \and Vincent Guedj}

\address{Department of Mathematics\\
         University of Notre Dame\\
         Notre Dame, IN 46556}
\email{diller.1@nd.edu}

\address{CMLS \\ \'Ecole Polytechnique \\ 91128 Palaiseau  \\France}
\email{dujardin@math.polytechnique.fr}
\address{Universit\'e Aix-Marseille 1,\\ LATP, CMI, Technop\^ole de Ch\^ateau Gombert \\ 
39 rue F.Joliot-Curie \\ 13453 MARSEILLE Cedex 13, \\  France}
\email{guedj@cmi.univ-mrs.fr}

\thanks{J.D. acknowldeges the National Science Foundation for its support through grant DMS 06-53678.}
\thanks{R.D. was partially supported by ANR through project BERKO}
\begin{document}

\selectlanguage{english}
\maketitle
\begin{abstract}
We continue our study of the dynamics of mappings with small topological degree on projective complex surfaces. Previously, under mild hypotheses, we have constructed an ergodic ``equilibrium'' measure for each such mapping.  Here we study the dynamical properties of this measure in detail: we give optimal bounds for its Lyapunov exponents, prove that it has maximal entropy, and show that it has product structure in the natural extension. Under a natural further assumption, we show that saddle points are equidistributed towards this measure. 
This generalizes results that were known in the invertible case  and adds to the small number of situations 
 in which a natural invariant measure for a non-invertible dynamical system is well-understood.
\end{abstract}
%
%

\selectlanguage{english}

\section*{Introduction}

In this article we continue our investigation, begun  in \cite{part1, part2}, of dynamics on complex surfaces for rational transformations with small topological degree.  Our previous work culminated in the construction of a 
canonical mixing invariant measure for a very broad class of such mappings.  We intend now to  study in detail the nature of this measure.  As we will show, the measure meets conjectural expectations concerning, among other things, Lyapunov exponents, 
entropy, product structure in the natural extension, and equidistribution of saddle orbits.
\medskip

Before entering into the details of our results, let us recall our setting. Let $X$ be a complex
projective surface (always compact and connected), and $f:X\to X$ be a rational mapping.  Our main requirement is that $f$ has \emph{small topological degree}: 
\begin{equation}
\la_2(f) < \la_1(f).
\end{equation}
Here the topological (or second dynamical) degree $\la_2(f)$ is the number of preimages
of a generic point, whereas the first dynamical degree $\la_1(f) :=\lim \norm{(f^n)^*\rest{H^{1,1}(X)}}^{1/n}$
measures the asymptotic volume growth of preimages of curves under iteration of $f$. 
We refer the reader to \cite{part1} for a more precise discussion of dynamical degrees. In particular it was observed there that the existence of maps with small topological degree imposes some restrictions on the ambient
surface: either $X$ is rational (in particular, projective), or $X$ has Kodaira dimension zero.

Let us recall that the ergodic theory of mappings with large topological degree ($\la_2>\la_1$) has been extensively 
studied, and that results analogous to our Theorems \ref{thm:complete} and \ref{thm:saddle} are true in this context
\cite{briend-duval, ds, g}. In dimension 1, all rational maps have large topological degree, and in this setting these results are due to \cite{lyubich, flm}. We note also that in the birational case $\lambda_2=1$, the main results of this paper are obtained in \cite{bls,bls2} 
(for polynomial automorphisms of 
$\cd$) in \cite{cantat} (for automorphisms of projective surfaces)  and in \cite{birat} (for general birational maps). Hence the focus here is on noninvertible mappings which, as the reader will see, present substantial additional difficulties.

We will work under two additional assumptions, which we now introduce.

\medskip

\noindent {\bf Good birational model.}
We need to assume that the linear actions $(f^n)^*$
induced by $f^n$ on cohomology are compatible with the dynamics, i.e.
\begin{equation}
 \tag{\bfseries H1} 
(f^n)^*=(f^*)^n,
\text{ for all } n \in \N.
\end{equation}
This condition, often called 
 ``algebraic stability'' in the literature, was first considered by Fornaess and Sibony \cite{fs94}. There is some evidence  that for any 
mapping $(X,f)$ with small topological degree, there should exist a birationally conjugate mapping $(\tilde X,\tilde f)$ that satisfies (H1).  Birational conjugacy does not affect 
dynamical degrees, so in this case we simply replace the given system $(X,f)$ with the ``good birational model'' $(\tilde X,\tilde f)$.

We observed in \cite{part1} that the minimal model for $X$ is a good birational model when $X$ has Kodaira dimension zero.  For rational $X$, there is a fairly explicit blowing up proceedure \cite{df} that produces a good model when $\lambda_2=1$.  More recently, Favre and Jonsson \cite{fj} have prove that each polynomial mapping of $\cd$ with small topological degree admits a good model on passing to an iterate.

Under assumption (H1), in \cite{part1}, we have constructed and studied  canonical invariant
currents $T^+$ and $T^-$. These are defined  by
$$
T^+=\lim \frac{c^+}{\la^n} (f^n)^* \omega
\text{ and }
T^-=\lim \frac{c^-}{\la^n} (f^n)_* \omega
$$
where $\omega$ is a fixed K\"ahler form on $X$ and $c^\pm$ are normalizing constants chosen so that in cohomology $\{T^+\} \cdot \{T^-\}=\set{\omega}\cdot\set{T^-} =\set{\omega}\cdot\set{T^+}=1$.   A fact of central importance to us is that these currents have additional geometric structures: $T^+$ is  laminar, while $T^-$ is woven (see below \S \ref{sec:geom} for definitions).

\medskip

\noindent {\bf Finite energy.}
Let $I^+$ denote the indeterminacy set of $f$, i.e. the collection of those points that $f$ ``blows up'' to curves; and let $I^-$ denote the analogous set of points which are images of curves under $f$.  The invariant current $T^+$ (resp. $T^-$) typically has positive Lelong number at each point of the extended indeterminacy set $I^+_{\infty}=\bigcup_{n \geq 0} f^{-n} I^+$ (resp. $I^-_{\infty}=\bigcup_{n \geq 0} f^{n} I^-$).
Condition (H1) is equivalent to asking that the sets $I^+_{\infty}$ and $I^-_{\infty}$ be disjoint. 

In order to give meaning to and study the wedge product $T^+\wedge T^-$, it is desirable to have more
quantitative control on how fast these sets approach one another.  This is how our next hypothesis should 
be understood:
\begin{equation}
 \tag{\bfseries H2} 
f \text{ has finite dynamical energy.}
\end{equation} 
We refer the reader to \cite{part2} for a precise definition of finite energy and its relationship with recurrence 
properties of indeterminacy points.  In that article we proved the following theorem.

\begin{theo}[\cite{part2}]\label{theo:ddg2}
 Let $f$ be a meromorphic map with small topological degree on a projective surface, satisfying hypotheses (H1) and (H2). Then  
the wedge product $\mu:=T^+ \wedge T^-$ is a well-defined probability invariant measure that is  $f$-invariant and mixing.  Furthermore the wedge product is described by the geometric intersection of the laminar/woven structures of $T^+$ and $T^-$.
\end{theo}

 The notion of ``geometric intersection'' will be described  at length in \S \ref{sec:geom}.

\medskip
\begin{center}$\diamond$\end{center}
\medskip

We can now state the main results of this article. Let us emphasize that they  rely on the hypotheses (H1) and (H2) only through the conclusions of Theorem \ref{theo:ddg2}.  Taking these conclusions as a starting point, one can read the proofs given here independently of \cite{part1, part2}.

\begin{theo}\label{thm:complete}
Let $X$ be a complex projective surface and $f: X \rightarrow X$ be a rational map with small topological degree. 
Assume that $f$ satisfies the conclusions of  Theorem \ref{theo:ddg2}. Then the canonical invariant measure
$\mu = T^+\wedge T^-$ has the following properties:
\begin{enumerate}[i.]
 \item For $\mu$-a.e. $p$ there exists a nonzero tangent vector $e^s$  at $p$, 
 such that
\begin{equation}\label{eq:es}
\limsup_{ n\cv\infty} \unsur{n} \log\abs{df^n(e^s(p))}\leq 
- \frac{\log(\lambda_1/\lambda_2)}{2}.
\end{equation}
 \item Likewise, for $\mu$-a.e. $p$ there exist a  tangent vector $e^u$  at
$p$, and a set of integers $\nn'\subset \nn$ of density 1  
such that 
\begin{equation}\label{eq:eu}
\liminf_{ \nn'\ni n \cv\infty} \unsur{n} \log\abs{df^n(e^u(p))}\geq 
 \frac{\log\lambda_1}{2}.
\end{equation}
\item $\mu$ has entropy $\log \lambda_1$; thus it has maximal entropy and $h_{\rm top}(f)= \log\lambda_1$.
\item The natural extension of $\mu$ has local product structure. 
\end{enumerate}
\end{theo}

In particular it follows from {\em iv.} and the work of Ornstein and
Weiss \cite{ow} (see also  Briend \cite{briend} for useful remarks on the
adaptation to the noninvertible case) that the natural extension of $\mu$ has the Bernoulli property,
hence $\mu$ is mixing to all orders and has the $K$ property. 
 A precise definition of local product structure will be given below in \S 
\ref{sec:product}. This is  the analogue of the balanced property of the maximal measure in the large topological degree case. 

\medskip

Let us stress that we do not assume that $\log \mathrm{dist}(\cdot,  I^+\cup C_f)$ is $\mu$ integrable ($C_f$ denotes the critical set). This condition is usually imposed to guarantee the existence of Lyapunov exponents and applicability of the Pesin theory of non-uniformy hyperbolic dynamical systems.  However, for mappings with small topological degree, it is known to fail in general (see \cite[\S 4.4]{part2})).  This contrasts with the large topological degree case, in which the maximal entropy measure integrates all quasi-psh functions.

When the Lyapunov exponents $\chi^+(\mu) \geq \chi^-(\mu)$ are well defined, then
{\em (i)} and {\em (ii)} imply that
$$
\chi^+(\mu) \geq \frac{1}{2} \log \la_1(f) >0
> -\frac{1}{2} \log \la_1(f)/\la_2(f) \geq \chi^-(\mu),
$$
hence the measure $\mu$ is hyperbolic. These bounds are optimal and were conjectured  in \cite{g05b}.

In order to go further and relate $\mu$ to the distribution of saddle periodic points, we use Pesin theory and must therefore invoke the above integrability hypothesis.

\begin{theo}\label{thm:saddle}
Under the assumptions of Theorem \ref{thm:complete}, assume further that
\begin{equation}
 \tag{\bfseries H3}  p\mapsto \log \mathrm{dist}(p, I^+\cup C_f)\in L^1(\mu),
\end{equation}
where $C_f$ is the critical set.

Then, for every $n$ there exists a set $\mathcal{P}_{n}\subset \supp(\mu)$ of saddle periodic points 
of period $n$, with $\#\mathcal{P}_{n}\sim \lambda_1^n$, and  such that
$$
\unsur{\lambda_1^n}\sum_{q\in \mathcal{P}_n} \delta_q\longrightarrow \mu.
$$

Let $\mathrm{Per}_n$ be the set of all isolated periodic points of $f$ of period $n$. If furthermore
\begin{itemize}
\item[-] $f$ has no curves of periodic points,
\item[-] or $X=\pd$ or $\pu\times \pu$,
\end{itemize}
then $\#\mathrm{Per}_n \sim \lambda_1^n$, so that asymptotically nearly all periodic points are saddles.
\end{theo}

This theorem was proved for birational maps by the second author in \cite{birat} (though the possibility of a curve of periodic points was overlooked there).  It would be interesting to prove a similar result without using Pesin Theory (i.e. without assumption (H3)). 

 It would also be interesting to know when saddle points might lie outside $\supp(\mu)$.  
One can easily create isolated saddle points
 by blowing up an attracting fixed point with unequal eigenvalues.   
We then get an infinitely near saddle point in the direction corresponding to the larger multiplier,
whose unstable manifold is contained in the exceptional divisor of the blow-up. 
We do not know any example of a saddle point outside $\supp(\mu)$ whose stable and unstable manifolds are both Zariski dense.  

\medskip

While the results in this paper parallel those in \cite{birat}, new and more elaborate arguments are needed for non-invertible maps. In particular, we are led to work in the natural extension (e.g. for establishing {\em iii.} and {\em iv.} of Theorem \ref{thm:complete} ), in a situation where there is no symmetry between the preimages along $\mu$ (see the examples in \S \ref{sec:idea}).
An interesting thing to note is that additional regularity for the potentials of the invariant currents does not help our arguments much.

\medskip

Under an assumption similar to (H3), De Th\'elin and Vigny \cite{dv} have recently found an alternate way to compute entropy, using Yomdin's Theorem \cite{y}. Together with the work 
\cite{dt-lyap}, this leads to an alternate proof of the bounds on Lyapunov exponents (again, \cite{dt-lyap} requires (H3)). An advantage of their method is that it works in higher dimension. On the other hand they do not obtain local product structure nor the equidistribution of saddle orbits. In any case, it seems that computing the entropy of the mappings considered here is always a delicate issue. 

 It is also worth noting that there are few cases in which natural invariant measures for non-invertible differentiable dynamical systems have well understood natural extensions.  Measures of maximal entropy for rational mappings with large topological degree (in one and several dimensions) afford one example.  Another \cite{ledrappier-srb} is given by absolutely continuous invariant measures for interval maps.

\medskip

An important remaining question  is that of uniqueness of the maximal entropy measure. Let us comment a little bit on this point. The first difficulty is that another candidate measure of maximal entropy needn't integrate $\log \mathrm{dist}(\cdot, I^+\cup C_f)$, so that it is delicate to work with. Therefore it is reasonable to restrict the uniqueness problem to measures satisfying this assumption. However, even with this restriction, and even with the additional assumption that $f$ is birational, the problem remains unsolved.

\medskip
\begin{center}$\diamond$\end{center}
\medskip

We now discuss applicability of our assumptions.  For mappings on Kodaira zero surfaces, we have seen that we can always assume that (H1) is satisfied. It is actually not very hard to prove that (H2) and (H3) are also always true (see \cite[Proposition 4.8]{part1} and 
\cite[Proposition 4.5]{part2}). Thus Theorems \ref{thm:complete} and \ref{thm:saddle}  yield the following:

\begin{coro}
Let $X$ be a complex projective surface of Kodaira dimension zero.
Let $f: X \rightarrow X$ be a rational transformation with small
topological degree. Then the conclusions of Theorems \ref{thm:complete} 
and \ref{thm:saddle} hold for $f$. 
\end{coro}

When $X$ is rational our results apply notably to the case where $f$ is the rational extension of a polynomial mapping of $\C^2$ with small topological degree. As noted above, Favre and Jonsson have proven that a slightly weaker variation of (H1) holds for $f$ in a suitable compactification $X$ of $\C^2$.
As we show in \cite[\S 4.1.1]{part1}, this variation is sufficient for our purposes. 
More precisely, it is explained there that though (H1) holds only for  some iterate $f^k$, 
the currents $T^+$ and $T^-$ are actually invariant by  $f$.  Since (H2) depends on $f$ only through the currents $T^+$ and $T^-$, and  (H3) holds for $f$ as soon as it holds for an iterate,  we conclude from  \cite{part2}\footnote{\cite{part2} was accepted before \cite{part1}, 
so statements about  polynomial maps in that article are still given in terms of  iterates.}  that (H2) and (H3), hence
the conclusions of Theorem \ref{theo:ddg2}, hold for $f$. 
Altogether this implies the following corollary.

\begin{coro}
Let $f: \C^2 \rightarrow \C^2$ be a polynomial mapping with small topological degree. Then the conclusions of Theorems \ref{thm:complete} 
and \ref{thm:saddle} apply to $f$. 
\end{coro}

Let us emphasize that the small topological degree assumption is needed here:
the reader will find in \cite{g05b} an easy example of a polynomial endomorphism
$f:\C^2 \rightarrow \C^2$ with $\la_1(f)=\la_2(f)>1$ such that
$h_{top}(f)=0$.

\medskip

Regarding more general rational mappings on rational surfaces, we noted above there are grounds 
to suspect that (H1) holds generally, after suitable birational conjugation.  Among maps satisfying (H1), there is no known example of a map that violates the energy condition (H2).  We offer some evidence in \cite{part2} that (H2) might only fail in very degenerate situations.  
On the other hand, we give examples \cite[\S 4.4]{part2} of mappings satisfying (H1) and (H2) but not (H3). Nevertheless, it seems plausible that (H3) is generically satisfied (see \cite[Prop. 4.5]{bedford-diller}).

\medskip
\begin{center}
 $\diamond$
\end{center}
\medskip

The structure of this paper is as follows.  \S \ref{sec:geom} is devoted to 
some (mostly non-classical) preliminaries  on geometric currents, while \S 
\ref{sec:ergodic}  recalls some well-known facts from Ergodic Theory.
The proof 
of Theorem \ref{thm:complete} occupies \S \ref{sec:idea} to \S \ref{sec:product}  (we give a more precise plan of the proof in 
\S \ref{sec:idea}). Finally \S \ref{sec:saddle} is devoted to 
the proof of  Theorem \ref{thm:saddle}.

\medskip

\noindent{\bf Acknowledgments.} We are grateful to the referee for a very careful reading and helpful comments.
 R.D. also  thanks J.-P. Thouvenot for useful conversations.

\section{Preliminaries on geometric currents}\label{sec:geom}

We begin by collecting some general facts about geometric, that is laminar and
woven, currents.  We often use the single word ``current'' as a shorthand for 
``positive closed (1,1) current  on a complex surface.'' 

\subsection{Laminations and laminar currents} 
Recall that a {\em lamination by Riemann surfaces} is a topological space such that every point
admits a neighborhood $U_\alpha$  homeomorphic (by $\phi_\alpha$)
to a product of the form $\dd\times \tau_\alpha$ (with coordinates $(z,t)$), 
where $\tau_\alpha$ is some locally compact set, $\dd$ is the unit disk, 
and such that the transition maps $\phi_\alpha\circ\phi_\beta^{-1}$ are 
 of the form $(h_1 (z,t), h_2(t))$, with $h_1$ holomorphic in the disk direction $z$. By definition a
{\em plaque} is a subset of the form $\phi_\alpha^{-1}(\dd\times\set{t})$, and  a {\em flow box} 
is a subset of the form $\phi_\alpha^{-1}(\dd\times K)$, with $K$ a compact set. A {\em leaf} is a minimal connected set $L$
with the property that every plaque intersecting $L$ is contained in $L$.
An {\em invariant transverse measure} is  given by 
a collection of measures on the transverse sets $\tau_\alpha$, 
compatible with the transition maps $\phi_\alpha\circ\phi_\beta^{-1}$. 
The survey by Ghys \cite{ghys} is a good reference for these notions.
We always assume that the space is separable, so that is is covered by
countably many flow boxes. In this paper we will consider ``abstract'' laminations
 as well as  laminations embedded in complex surfaces. In the latter 
case we require of course that the complex structure along the plaques is compatible with the ambient one.

Two flow boxes embedded in a manifold are said to be {\em compatible} if the corresponding plaques 
intersect along open sets. Notice that disjoint flow boxes are compatible by definition.
A {\em weak lamination} is a countable union of compatible flow boxes.  It makes
sense to speak of leaves and  invariant transverse measures on a weak lamination. 
Being primarily interested in measure-theoretic  properties, 
we  needn't distinguish between ordinary and weak laminations in this paper.  
%

\medskip

Let us also recall that a $(1,1)$ current $T$ is {uniformly laminar} if it is given by integration over an embedded lamination endowed with an invariant 
transverse measure. That is, the restriction $T\rest{\phi^{-1}(\tau\times\dd)}$ to a single flow box can be expressed
$$
\int_\tau [ \phi^{-1}(\set{t}\times\dd)] d\mu_{\tau}(t),
$$ 
where $\mu_\tau$ is the measure induced by the transverse measure on $\tau$.

The current $T$ is {\em laminar} if it is an integral over a measurable family of compatible holomorphic disks.  
Equivalently, for each $\e>0$ there exists an open set  $X_\e\subset X$ and a uniformly laminar current $T_\e\leq T$ in $X_\e$ 
such that the mass (in $X$) of the difference satisfies $\m(T-T_\e)<\e$.
It is a key fact that the laminar currents we consider in this paper have some additional geometric properties.  
For instance, each has a natural underlying weak lamination, and the lamination carries an invariant transverse measure. 
We refer the reader to \cite{structure, birat} for details about this.

Our main purpose in this section is to explore 
the related notion of {\em woven current} and  
generalize some results of  \cite{structure} that we will need afterwards.

\subsection{Marked woven currents.} Given an open set $Q$ in a Hermitian complex surface, we let $Z(Q,C)$ denote the set of (codimension 1) analytic chains with volume
bounded by $C$.  We endow $Z(Q,C)$ with the topology of currents.  Since there is a
dense sequence of test forms, this topology is metrizable. Most often we identify a chain and its support, which is a closed analytic subset of $Q$.
  We denote by $\Delta(Q,C)\subset Z(Q,C)$ the
closure of the set of analytic disks in $Z(Q,C)$. 

By definition, a {\em uniformly woven current}  in $Q$ is an integral of
integration currents over chains in $Z(Q,C)$, for some $C$ \cite{dinh}. A 
{\em  woven current} is an increasing limit of sums of uniformly woven currents.

\medskip

A given laminar current can be expressed as an integral of disks in an
essentially unique fashion (only reparameterizations up to a set of zero 
measure are possible \cite[Lemma 6.5]{bls}). For woven 
(even uniformly) currents, this is not true.  For example,
$$\omega= idz\wedge d\bar z + idw\wedge d\bar w = \frac12 id(z+w)\wedge
d\overline{( z+w)} + 
\frac12 id(z-w)\wedge d\overline{( z-w)}.$$ 
That is, the standard K\"ahler form in $\cd$ may be written as a sum of
uniformly laminar currents in two very different ways.

We say that a uniformly woven current is {\em marked} if it is presented as an
integral of disks.  More specifically, a marking for a uniformly woven current
$T$ in some open set $Q$ is a positive Borel measure ${m(T)}$
on $Z(Q,C)$ for some $C$ such that 
$$
T= \int_{Z(Q,C)}[D] d{m(T)}(D).
$$
Abusing terminology, we will also refer to ${m(T)}$ as the {\em transverse
measure} 
associated to $T$.  We call the support of $m(T)$ the {\em web supporting} $T$.
The woven currents considered in this paper have the additional property of being
 {\em strongly approximable} (see \S\ref{subs:markings} below for the definition). This allows us to 
show (Lemmas \ref{lem:graph} and \ref{lem:chain}) 
that their markings are concentrated on 
$\Delta(Q,C)$. 

We define \textit{strong convergence} for marked uniformly woven currents as follows: 
 $T_n$ strongly converges to  $T$ 
if the markings $m(T_n)$ are supported in $Z(Q,C)$ for a fixed $C$ and converge
weakly 
to ${m(T)}$. We leave the reader to check that this implies the usual
convergence of currents.

\subsection{Markings for strongly approximable woven currents.}\label{subs:markings}

The woven structure for the invariant current $T^-$ has several additional properties, which play a crucial role in the paper.
We say that a current $T$ is a {\em strongly approximable woven current} if it is obtained 
as a limit of divisors $[C_n]/d_n$ whose geometric genus is  $O(d_n)$
(here by definition the  geometric genus of a chain is the sum of the genera of its components). See \cite[\S 3]{part1} for a proof that $T^-$ is of this type.

Its woven structure can be  constructed as follows.   We fix two generically 
transverse linear projections $\pi_i:X\cv\pu$, and subdivisions by squares of the
projection bases.  For each square $S\subset \mathbb{P}^1$, we
discard from the approximating sequence $[C_n]/d_n$ all connected components
$\pi_i^{-1}(S)\cap C_n$ over $S$ which are not graphs of area $\leq 1/2$. The
geometric assumption on $C_n$ implies that the corresponding loss of mass is
small.  Taken together, the two projections divide the ambient manifold into a
collection $\qq$ of ``cubes'' of size $r$ and these partition the remaining well-behaved
part of $[C_n]/d_n$ into a collection of uniformly woven currents $T_{Q,n}$
whose sum $T_{\qq,n}$ closely approximates $[C_n]/d_n$. The
 disks constituting $T_{\qq,n}$ will be referred to as ``good components''.

More specifically 
 we define a \emph{cube} to be a subset of the form $Q:= \pi_1^{-1}(S_1)\cap
\pi_2^{-1}(S_2)$, where the $S_i$ are
squares. Near a  point where $\pi_1$ and $\pi_2$ are regular and the squares $S_i$ are small enough,  $Q$  is actually biholomorphic to an affine cube.  
The woven current $T_{Q,n}$, or more precisely its marking $m(T_{Q,n})$, is defined by
assigning mass $1/d_n$ to each `good' component of $C_n\cap Q$. Then we have the mass estimate  $\m\big(T_{\qq,n}-[C_n]/d_n\big)=O(r^2)$, 
independent of $n$, where $r$ is the size of the cubes.

There is a subtle point here.  The number of disks constituting $C_n\cap Q$
might be much larger than $d_n$.  Thus the masses of the measures $m(T_{Q,n})$
might not be bounded above uniformly in $n$.  However, by Lelong's theorem
\cite{lelong} the disks intersecting a smaller subcube $Q'$ have volume bounded
from below, so there are no more than $\leq c(Q', Q)d_n$ of these. 
Restricting the marking  to these disks gives rise to a new current, which we continue to
denote 
by $T_{Q,n}$, that coincides with the old one on $Q'$. The
mass of $m(T_{Q,n})$ is now locally uniformly bounded in $Z(Q, 1/2)$, and we may
extract a convergent subsequence, letting $m(T_Q)$ denote the limiting measure
and $T_Q$ the corresponding current. Let $T_\qq$ be the sum of these currents, where  $Q$ ranges over the subdivision $\qq$. We then  have that 
$\m\big(T_{\qq}-T\big)=O(r^2)$. 

Finally, we consider a increasing sequence of subdivisions by cubes $\qq_i$ of size $r_i\cv 0$, and by the previous procedure we
obtain an increasing sequence of currents  $T_{\qq_i}$, converging to $T$ by the previous estimate. Notice that in the case of $T^+$ the construction is the same, except that at each step the disks constituting the $T_{\qq,n}$ are disjoint so that  the $T_\qq$ are uniformly laminar.

\medskip

Since the approximating disks are restrictions of graphs over each projection,
the marking of $T_Q$ has the following 
additional virtues.

\begin{lem}\label{lem:graph}
The multiplicity with which disks in $\supp m(T_{Q,n})$ converge to chains belonging to
$\supp m(T_Q)$ is always equal to one.
Similarly, if a chain $D\in \supp(m(T_Q))$ is a Hausdorff limit of other chains
$D_n\in \supp(m(T_Q))$, then the multiplicity of convergence is equal to one. 
Thus there is no folding; i.e. in either cases the corresponding tangent spaces also converge.  
In particular the chains in $\supp(m(T_Q))$ are non-singular.
\end{lem}

On the other hand, the Hausdorff limit of a sequence of disks might not be
itself a disk but rather a chain with several components.  As an example,
consider the family of parabolas $w=z^2+c$, $c\in\mathbf{C}$, each restricted to
the (open) unit bidisk. For $\abs{c}<1$ the restriction has a single simply
connected component. But when $\abs{c}\to 1$, it becomes disconnected.
The following observation will be useful to us in several places.

\begin{lem} \label{lem:chain}
 For generic subdivisons by squares of the projection bases of $\pi_1$ and
$\pi_2$, we have that $m(T_Q)$-almost every chain is a disk (i.e. has only one
component).
\end{lem}

\begin{proof}
 Recall that each chain in the support of $m(T_{Q,n})$ is obtained by
intersecting a graph over some square in the base of e.g. $\pi_1$ with the
preimage $\pi_2^{-1}(S)$ of a square in the base of $\pi_2$. So if a sequence of
disks has a disconnected limit, then the limit must be a piece of a graph tangent
to  some fiber of $\pi_2$.  However, a graph over $\pi_1$ that is not outright
contained in a fiber of $\pi_2$ will be tangent to at most countably many fibers
of $\pi_2$. Since we have uncountably many choices for the subdivision  by
squares associated to $\pi_2$, the result follows from standard measure theory
arguments.
\end{proof}

\subsection{Geometric intersection.} \label{subs:geom}
Let $Q\subset \cd$ be an open subset, and 
$D$, $D'$ be  two holomorphic chains in $Q$. We define 
$[D]\geom [D']$ as the sum of point masses, counting multiplicities, at isolated
intersections of $D$ and $D'$.    
Given more generally two marked uniformly woven currents $T_1$ and  $T_2$, with associated
measures $m_1$ and $m_2$,
we define the {\em geometric intersection} as 
$$ 
T_1\geom T_2 = \int [D_1]\geom[D_2] \left(dm_1\otimes dm_2\right)
(D_1, D_2).
$$
In general it is necessary to take multiplicities into account because
of the possibility of persistently non transverse intersection of
chains. Notice that the definition
depends not only on the currents but also on the markings. 

The following basic proposition says that under reasonable assumptions the wedge
product of uniformly woven currents is geometric.

\begin{pro}\label{prop:L1}
If  $T_1$ and  $T_2$ are as above and if  $T_1\in L^1_{\rm loc}(T_2)$,
then 
$T_1\wedge T_2 = T_1\geom T_2$.
\end{pro}

See  \cite[Prop. 2.6]{part2} for a proof. 
As a corollary, if $T_1\in L^1_{\rm loc}(T_2)$,
the geometric wedge product $ T_1\geom T_2$ is independent of the markings.

\medskip

Since woven currents are less well-behaved than laminar ones, we do not try as in the laminar case \cite{birat} to
give an intrinsic definition of the geometric intersection of woven currents. 
Instead we focus only on the particular situation that arises in this paper: we have strongly approximable 
currents  $T^+$ (laminar) and $T^-$ (woven) in $X$, whose geometric structures are
obtained by extracting good components of approximating curves, as explained in the
previous subsection.  Let $\qq_i$ be the increasing sequence of subdivisions by cubes constructed as in \S \ref{subs:markings}, 
and let $T_{\qq_i}^\pm$ be the corresponding currents. 
Let us write $T^+_\qq\geom T^-_\qq$ for the sum of $T^+_Q\geom T^-_Q$ over $Q\in \qq$. 

Under the finite energy hypothesis (H2) from the introduction,  we have shown in 
 \cite{part2}  that the wedge product 
 $T^+\wedge T^-$ is a well-defined probability measure.  
Though the currents $T_{\qq_i}^\pm$ depend on the choice of generic linear projections, of generic subdivisions for the projection bases\footnote{Here genericity is understood in the measure-theoretic sense.} and of convergent subsequences extracted from $T_{\qq_i,n}^\pm$, the central result of \cite[\S 2.3]{part2} is that the measures $T^+_{\qq_i}\geom T^-_{\qq_i}$ 
increase to $T^+\wedge T^-$ regardless.  We summarize the facts that the limiting measure has the correct mass and is independent of choices by saying that {\em the intersection of $T^+$ and $T^-$ is geometric}. 

A slightly delicate point is that in \cite{part2}, $T^+\wedge T^-$ is not always defined in the usual $L^1_{\rm loc}$ fashion. 
In particular it is not clear whether $T^+_Q\in L^1_{\rm loc}(T^-_Q)$ for  $Q\in \qq_i$. Nevertheless 
\cite[Lemma 2.10]{part2} implies that this is true after negligible modification of the markings. So throughout the paper we assume that for every $i$ and every $Q\in \qq_i$, $T^+_Q\in L^1_{\rm loc}(T^-_Q)$ so that by Proposition \ref{prop:L1} we can always identify 
$T^+_Q\wedge T^-_Q$ and $T^+_Q\geom T^-_Q$.

\medskip

We now prove  the useful fact that 
the  geometric product of uniformly woven currents is   lower 
semicontinuous with respect to the strong topology. It can further be seen that,
except for boundary effects, 
discontinuities can occur only when the chains have common components. 

\begin{lem}\label{cor:lsc} Let $Q\Subset {Q'}$ be open sets and $T_1,T_2$ be marked,  uniformly woven currents in $Q$ with
markings supported on $\Delta(Q',1/2)$.  Then 
$$
(T_1,T_2)\longmapsto \int_{Q}T_1\geom T_2
$$
is lower semicontinuous with respect to the strong topology.
\end{lem}

\begin{proof} Observe first that for $D,D'\in \Delta(Q',1/2)$, the mass $\int_Q [D]\geom [D']$ on $Q$ is lower semicontinuous in the strong topology. This is just a fancy way of saying that isolated intersections between $D$ and $D'$ persist under small perturbation.

Now for the general case, suppose that $T^j_1, T^j_2$ strongly converge to $T_1, T_2$, and let $m^j_1,m^j_2$ be the corresponding transverse measures. Then 
$$
\int_Q T^j_1\geom T^j_2 = \int_{\Delta(Q',1/2)^2} \left(\int_Q [D]\geom [D']\right)
d( m^j_1\otimes  m^j_2).
$$
Now by strong convergence, we have that $m_1^j\otimes  m^j_2$ converges to $ m_1\otimes m_2$ weakly; and we have just seen that the inner integral is lower semicontinuous in $D,D'$.  Hence the lemma follows from a well-known bit of measure theory: if $\nu_j$ is a sequence of positive Radon measures with uniformly bounded masses, weakly converging to some $\nu$, and if $\varphi$ is any lower semicontinuous function, then
$\liminf_{j\cv\infty} \langle\nu_j, \varphi\rangle \geq \langle\nu ,
\varphi\rangle$.
\end{proof}

\medskip

We will also need the following lemma.  

\begin{lem}\label{lem:nonatomic}
If $L$ is a generic 
hyperplane section of $X$, the wedge product $T^+\wedge [f^k(L)]$ is well defined 
for all $k \in \N$
(in the $L^{1}$ sense) and gives no mass to points. 
\end{lem}

\begin{proof}
Since $[f^k(L)] = f^k_*[L]$ for any $L$ disjoint from $I(f^k)$, it suffices by invariance of $T^+$ to prove the result for $k=0$.
Fix a projective embedding $X\hookrightarrow \bfP^\ell$.
Then the hyperplane sections on $X$ are parametrized by the dual $\bfP^{\ell*}$.  If $dv$ is Fubini-Study
volume on $\bfP^{\ell*}$, then the Crofton formula says that $\alpha := \int_{\bfP^{\ell*}} [L]\,dv(L)$ is the restriction to $X$ of the Fubini-Study K\"ahler form on $\bfP^\ell$.  
Hence $T^+\wedge[L]$ is well-defined for almost all $L$ and $T^+\wedge\alpha = \int T^+\wedge[L]\, dv(L)$.

To show  
that $T^+\wedge [L]$ does not charge points for generic $L$, we exploit laminarity\footnote{This fact 
 does not seem to follow trivially from the fact that $T^+$ does not charge curves} . Since $\alpha$ is smooth, 
we have that $T^+_{\qq_i}\wedge \alpha$ increases to $T^+\wedge\alpha$ as 
$i\cv\infty$. Thus for  almost any $[L]$, $T^+_{\qq_i}\wedge [L]$ 
increases to $T^+\wedge[L]$.  In particular, $T^+_{\qq_i}\wedge[L]$ is well-defined, and since
$T^+_{\qq_i}$ is a sum of uniformly laminar currents, the intersection is geometric. Since $T^+$ does not charge curves, $T^+_{\qq_i}$ is diffuse and $T^+_{\qq_i}\geom [L]$ therefore gives no mass to points.  Hence neither does $T^+\wedge[L]$.
\end{proof}

\subsection{The tautological bundle.}  Let $T$ be a marked uniformly woven
current in $Q$, with associated 
measure $ m(T)$ as above. The {\em tautological bundle} over $Z(Q,C)$ is the
(closed)  set 
$$\check{Z}(Q,C) = 
\set{(D, p)  ,\ p\in \supp(D) }\subset  Z(Q,C)\times Q.$$ 
Similarly we define the tautological bundle $\check T$ 
over $T$ by restricting to 
$\supp  m(T)$. Defining the tautological bundle is a somewhat artificial way of separating the disks of $T$,  which will nevertheless be quite useful conceptually.
 
In particular, when passing to the 
tautological bundle, the web supporting $T$ becomes a weak lamination with
transverse measure $ m(T)$. 
When (as in our situation) there is no folding, we get a lamination.

Let $\check\sigma_T$ be the product of the area measure along the 
disks $D$
with $m(T)$. 
Let $\check{\pi}:\check{T}\cv Q$ be the natural projection.  Then $\check\pi_*
\check\sigma_T$ is the usual 
trace measure $\sigma_T$ of $T$.  For $\sigma_T$  a.e. $p\in Q$,
$\check\sigma_T$ induces a conditional measure $\check\sigma_T(\cdot|p)$ on the
fiber $\check\pi^{-1}(p)$, which records as a measure the set of disks passing
through 
$p$.

\subsection{Analytic continuation property.} The dynamical results in \cite{birat} depend on fine properties of strongly approximable uniformly laminar currents proved in \cite{structure}. The situation is similar here.  Therefore we now state and sketch the proof of an ``analytic continuation'' property for curves subordinate to strongly approximable woven currents.  This is a fairly straightforward extension of \cite{structure} (see in particular Remark 3.13 there).  Hence our brief account closely follows the more detailed presentation in \S 3 of that paper.  

Recall our notation from \S \ref{subs:markings}: 
we have a normalized sequence $[C_n]/d_n$ of curves with slowly growing genera converging to our current $T$, a generic linear projection $\pi:X\to\pu$, and
approximations $T_{\qq_i,n}$, $T_{\qq_i}$ of $[C_n]/d_n$ and $T$ corresponding to a sequence of subdivisions $\qq_i$ of $\pu$ into squares $Q$.  Note that in order to better cohere with notation in \cite{structure}, we are departing from our overall convention by letting $Q$ denote a {\em square} in $\pu$  
rather than a cube in $X$.
Likewise, the notation $T_Q$ (resp. $T_{\qq}$, etc.) refers to a current in   
$\pi^{-1}(Q)$ that is a limit of families of good components of $C_n$ over $Q$. 
We may assume that each fiber of $\pi$ has unit area and intersects $C_n$ in $d_n$ points counting multiplicity.

Let $\mathcal{W}$ be a \emph{web}, i.e. an arbitrary union of smooth curves in some open set\footnote{We will only need the statement for webs of smooth curves, but the case of singular $\mathcal{W}$ could be of interest.}.  For any square $Q\subset\pu$, we define the restriction $T_Q\rest{\mathcal{W}}$ by restricting the marking measure 
for $T_Q$ to the set of disks contained in $\mathcal{W}$.  We then take $T_{\qq_i}\rest{\mathcal{W}} = \sum_{Q\in\qq_i} T_Q\rest{\mathcal{W}}$ and define $T\rest{\mathcal{W}}$ to be the increasing limit of $T_{\qq_i}\rest{\mathcal{W}}$.  Thus $T\rest{\mathcal{W}}$ depends on the choice of marking.  One can check, however, that it does not depend on the sequence of subdivisions $\qq_i$.

Here is the analytic continuation statement we will need.

\begin{thm}\label{thm:cont}
 $T\rest{\mathcal{W}}$ is a uniformly woven current.
\end{thm}

\begin{proof}[Proof (sketch)]  
We assume without loss of generality that the leaves 
$\Gamma_\alpha$ of $\mathcal{W}$ are graphs over a disk $U$ of area less than $1/2$.  Note that this allows us to view a restriction $T_Q\rest{\mathcal{W}}$ as the restriction $T_{Q,\mathcal{W}}\rest{Q}$ of a marked uniformly woven current $T_{Q,\mathcal{W}}$ defined over all of $U$.  Moreover, rather than work with a single sequence of subdivisions, we choose three such sequences $(\qq^j_i)_{i\in \nn}$, $j=1,2,3$ so that the set of all squares forms a neighborhood basis for $\pu$.  So given any $x\in \pu$, we can choose squares $Q$ decreasing to $x$ and define the marked uniformly woven current $T_{x,\mathcal{W}}$ to be the increasing limit of the currents $T_{Q,\mathcal{W}}$.
Given another uniformly woven current $S = \int \Gamma_\alpha\,ds(\alpha)$ supported on $\mathcal{W}$, we will say that $T$ \emph{strongly dominates} $S$ over $Q$ (resp. over $x$) if the marking for $T_{Q,\mathcal{W}}$ (resp. for $T_{x,\mathcal{W}}$) dominates $ds$.

Let us also recall the notion of ``defect''.  If $Q\subset\pu$ is a square, then the defect $\mathrm{dft}(Q,n)$ is the fraction of those components of $C_n$ over $Q$ that are bad.  By strong convergence, $\mathrm{dft}(Q,n)$ converges to a limit $\mathrm{dft}(Q)$ as $n\to\infty$.  Slow growth of genera implies that the total number of bad components over all squares in a subdivision is not larger than $Cd_n$.  Hence the total defect $\sum_{Q\in\qq_i} \mathrm{dft}(Q)$ of $T_{\qq_i}$ is bounded uniformly in $i$.   Then we can define the defect $\mathrm{dft}(x) = \lim_{Q\searrow x} \mathrm{dft}(Q)$ of $T$ over a point $x\in\pu$.  The total defect bound becomes $\sum_{x\in\pu} \mathrm{dft}(x) \leq 3C$.  In particular, $\mathrm{dft}(x)$ is positive for only countably many $x\in X$.  

The key fact underlying Theorem \ref{thm:cont} is (compare \cite[Prop. 3.10]{structure} and also \cite{duval})

\begin{lem}
Suppose that $x_0,x_1\in \pu$ are points satisfying $\mathrm{dft}(x_0) = \mathrm{dft}(x_1) = 0$.  Then $T_{x_0,\mathcal{W}} = T_{x_1,\mathcal{W}}$ as marked uniformly woven currents.
\end{lem}

One proves this lemma by choosing a path $\gamma$ from $x_0$ to $x_1$ along which $T$ has zero defect at every point.  Given $\e>0$, one can then choose a finite cover of $\gamma$ by squares $Q_0\ni x_0, Q_1, \dots , Q_N\ni x_1$ such that $\sum \mathrm{dft}(Q_j) < \e$, and such that the mass of $T_{Q_0,\mathcal{W}}$ is within $\e$ of that of 
$T_{x_0,\mathcal{W}}$. 
Beginning with $T_{Q_0,\mathcal{W}}$ and $Q_0$, one then proceeds from $Q_0$ to $Q_1$ and so on, keeping only that part of $T_{Q_0,\mathcal{W}}$ supported on leaves that remain good over the new square.  At each step one loses mass proportional to the defect over the squares involved.  In the end, one arrives at $Q_N$ with a marked uniformly woven current $S$ strongly dominated by each of the $T_{Q_j,\mathcal{W}}$ but with the mass bound $\m(T_{Q_0,\mathcal{W}}-S) < C\e$.  Letting $\e$ decrease to zero, one then infers that $T_{x_1,\mathcal{W}}$ strongly dominates $T_{x_0,\mathcal{W}}$.  By symmetry, we have equality.

From the lemma, we have a marked uniformly woven current $T_\mathcal{W} = T_{x,\mathcal{W}}$ that is independent of which zero defect point we use to define it.  Since each square contains points with zero defect, we have that $T_\mathcal{W}$ strongly dominates $T_{Q,\mathcal{W}}$ for every $Q$.  In particular $T|_\mathcal{W} \leq T_\mathcal{W}$.  On the other hand, the argument used to prove the lemma gives a slightly different statement: \emph{if $Q$ is a rectangle with $\mathrm{dft}(Q) < \delta$, then there is a marked uniformly woven current $S$ strongly dominated by $T_{Q,\mathcal{W}}$ such that $\m(T_\mathcal{W} - S) < C\delta$.}  To prove this, one applies the argument from the previous paragraph to the trivial path from some zero defect point $x_0 = x_1 \in Q$ ``covered'' by two squares $Q_0,Q_1$ such that $Q_1 = Q$ and $\mathrm{dft}(Q_0)$ is arbitrarily small.  

Now we can choose a finite number of points $x_1,\dots,x_N\in\pu$ such that $\sum_{x\neq x_j} \mathrm{dft}(x) < \e$.  Since neither $T$ nor $T_\mathcal{W}$ charge fibers of $\pi$, the modified version of the lemma tells that $\sum_{Q \in \qq_i^j, x_j\notin Q} T_Q\rest{\mathcal{W}}$, which is a woven current dominated by both $T\rest{\mathcal{W}}$ and $T_\mathcal{W}$, is within mass $C\e$ of $T_\mathcal{W}$ when $i$ is large.  Shrinking $\e$, we infer $T|_\mathcal{W} = T_{\mathcal{W}}$ as desired.
\end{proof}

\section{Preliminaries on ergodic theory}\label{sec:ergodic}

We now collect some standard facts from measurable dynamics that will be useful
to us.

\subsection{The natural extension.} 
A good reference for this paragraph is the (yet unpublished) book of M. Urbanski
and F. Przytycki \cite[Chapter 1]{PU}.

The natural extension of a (non-invertible) measurable dynamical system $(X,\mu,f)$ is the
(unique up to isomorphism) invertible system $(\hat X, \hat\mu, \hat f)$ semiconjugate to $(X,\mu,f)$ by a 
projection $\hat\pi:\hat X\cv X$ with the universal property that any other semiconjugacy
$\varpi:Y\to X$ of an invertible system $(Y,\nu,g)$ onto $(X,\mu,f)$ factors through $\hat\pi$.

The natural extension $\hat X$ may be presented concretely as the space of {\em histories}, i.e.  sequences $(x_n)_{{n\leq 0}}$ such that $f(x_n)=x_{n+1}$.  Here $\hat\pi = \pi_0$ is the projection 
$\left((x_n)_{n\leq 0}\right) = x_0$ onto the $0^\mathrm{th}$ factor; $\hat f$ is the shift map
$(x_j)\mapsto (x_{j+1})$; and $\hat{\mu}$ is the unique $\hat f$-invariant measure such that 
$\hat{\mu}({\pi}_0^{-1}(A)) = \mu(A)$.  From this point of view, the 
factorization of a semiconjugacy $\eta:Y\to X$ by some other invertible system
is easy to describe.  For each $y\in Y$, we have $\varpi(y) = \hat\pi\circ \eta(y)$, where 
$\eta(y)$ is the sequence $(\varpi( g^n(y)))_{{n\leq 0}}$.

The natural extension preserves entropy: $h_{\hat\mu}(\hat f) = h_\mu(f)$. Also
$(\hat f, \hat \mu)$ is ergodic iff $(f,\mu)$ is.

\medskip

Another charaterization of $\hat\mu$ is the following. 
It corresponds to the presentation of $(\hat X, \hat\mu, \hat f)$ as the inverse
limit of the system of measure preserving maps $(\cdots X \overset{f}{\cv} X
\cdots)$. Consider the standard model of the natural extension, and denote by
$\pi_{-n}$ the projection on the $({-n})^{\mathrm th}$ factor.

\begin{lem}\label{lem:inverse} 
If $\nu$ is any
 probability measure on $\hat X$ such that for every $n\geq 0$, $(\pi_{-n})_*\nu
= \mu$, then $\nu=\hat\mu$. 
\end{lem}

\begin{proof}
Let $C_{A_{-n},\ldots, A_{0}} = \big\{\hat x\in \hat X, \ \forall i\leq n, \
x_{-i} \in A_{-i}\big\}$ be a cylinder of depth $n$. One verifies easily that
under the assumption of the lemma, $\nu (C_{A_{-n},\ldots, A_{0}} ) \leq \hat\mu
(C_{A_{-n},\ldots, A_{0}} )$. From this we infer that $\nu \leq \hat\mu$, hence
$\nu=\hat\mu$ by equality of the masses.
\end{proof}

\subsection{Measurable partitions and conditional measures.} We will use the
formalism of measurable partitions and conditional measures, so we recall a few
facts (see \cite{bls} for a short presentation, and \cite{rokhlin, PU} for a
more systematic treatment). Recall first that a {\em Lebesgue space} is a
probability space isomorphic to the unit interval with Lebesgue
measure, plus countably many atoms.
 All the spaces we will consider in the paper are Lebesgue. A {\em
measurable partition} of a Lebesgue space is the partition into the fibers of
some measurable function. If $\xi$ is a measurable partition, a probability 
measure $\nu$ may be {\em disintegrated} with respect to $\xi$, 
giving rise to a probability measure 
$\nu_{\xi(x)}$ on almost every atom of $\xi$ (the conditional measure). 
The function 
$x\mapsto \int  \phi(y) d \nu_{\xi(x)}(y)$  is measurable and we have 
 the following disintegration
 formula: for every continuous function 
$\phi$, $$\int \left(\int   \phi(y) d \nu_{\xi(x)}(y) \right) d\nu(x) = \int\phi
d\nu.$$ Conversely, the validity of this  formula  for all $\phi$ characterizes
the conditional measures. 

Given partitions
$\xi_i$, we denote by $\bigvee \xi_i$ the {\em joint partition}, i.e.  
 $\left(\bigvee \xi_i\right)(x)  = \bigcap (\xi_i(x))$.

\medskip

If $\pi$ is a measurable map (possibly between different spaces) 
and $\xi$ a measurable partition, we  define the (measurable) partition
$\pi^{-1}\xi$ by $(\pi^{-1}\xi)(p)=\pi^{-1}(\xi(p))$.
We have the following easy lemma, whose proof is left to the reader.

\begin{lem}\label{lem:project}
Let $(\tilde Y,\tilde \nu)$, $(Y,\nu)$ be two probability spaces with a measure
preserving map $\pi:\tilde Y\cv Y$. Assume that $\xi$ is a measurable partition
of $Y$, and denote by $\tilde\xi$ the measurable partition $\pi^{-1}(\xi)$. 

Then for $\tilde \nu$-a.e. $p\in \tilde Y$,
$\pi_{*}(\tilde\nu_{\tilde\xi(p)})=\nu_{\xi(\pi(p))}.$
\end{lem} 

\section{Outline of proof of Theorem \ref{thm:complete}.}\label{sec:idea}

Before embarking to the proof, we give an overview of the main arguments leading
to Theorem \ref{thm:complete}. The proof of {\it i.} (\S \ref{sec:negative}) is based
on the study of the contraction properties along disks subordinate to the
laminar current $T^+$. This is delicate but fairly similar to \cite{birat},
and it is achieved in \S \ref{sec:negative}.

The proof of {\it ii.} is in the same spirit but with many more differences.
The fact that the current $T^-$ is only woven leads to substantial difficulties, 
the first of which is that there is no natural web associated to such a current.
This has been overcome by De Th{\'e}lin in \cite{dt-saddle} who has given a
short argument leading to {\em ii.}

Nevertheless to compute the entropy and establish local product structure we
need a finer analysis of $\mu$ and its natural extension. We therefore take a
longer path.  For the bounds on both positive and negative exponents, we use an
argument ``\`a la Lyubich'' that is suitable only for showing contraction.  So
for the positive exponent, we need to iterate backward.  To make this possible,
we first carefully select 
(\S \ref{sec:positive}) a set of {\em distinguished histories} which has
full measure in the natural extension, and exhibits  exponential contraction
along disks.  Then (\S \ref{sec:tautological}) we use the tautological
extension of $T^-$ to construct a dynamical system $(\check f, \check\mu)$ that
refines $(f,\mu)$
by making disks subordinate to $T^-$ disjoint. 
In particular, we may apply the
theory of measurable partitions and conditional measures to $\check f$.  
However, we cannot compute the entropy of $\check f$ using the Rokhlin formula as in
\cite{birat}, because constructing invariant partitions requires invertible
maps.  For entropy we consider finally the 
natural extension of $\check f$.  Happily, the natural extension of $\check f$
turns out to be isomorphic to that of $f$ 
(\S \ref{sec:entropy}; see also the figure on p.\pageref{figure:tautological} 
for a synthetic picture of the relationship between these spaces). 
Hence $h(f,\mu) = h(\check f, \check \mu)$. 
Once we know that $h(f,\mu) = \log\lambda_1$, the fact that $\mu$ is a measure of maximal entropy  follows from   Gromov's inequality \cite{gromov}. 

The product structure of the natural extension of $\mu$ follows from the above
analysis 
and the analytic continuation property of the disks subordinate to strongly
approximable woven/laminar currents (\S \ref{sec:product}).

\medskip

To understand some of the subtleties to come, we
encourage the reader to consider the following simple examples. 

\begin{exa}\label{ex:monomial}
 Let $f:\cd\cv\cd$ be a monomial map with small topological degree, that is of
the form 
$(z,w)\mapsto (z^aw^b,z^cw^d)$ where the matrix 
$A = \left(\begin{smallmatrix}
a& b\\ c & d\end{smallmatrix}\right)\in GL_2(\zz)$ satisfies 
$\abs{\mathrm{det}(A)} < \rho(A)$, e.g. $A = \left(\begin{smallmatrix} n & 1\\ 1 &
1\end{smallmatrix}\right)$ with $n\geq 3$. In this case the dynamics of $f$ is
well known and the unique measure of maximal entropy is
the Lebesgue measure on the unit torus $\set{\abs{z}=\abs{w}=1}$, which is a
{\em totally invariant} subset. Notice also that for this mapping, unstable
manifolds do not depend on histories.
\end{exa}

\begin{exa}\label{ex:henon}
 Let $f_0$ be a complex H\'enon map of degree $d\geq 3$, of the form $f_0(z,w)=
(aw+p(z), az)$. Consider now an integer $1<e<d$ and for  $b\ll a$ consider the
map $f_b(z,w)= (aw+p(z), az+b z^e)$. 
It is easy to prove that $f_b$ is algebraically stable in $\cd$, with 
$\lambda_1 = d$ and $\lambda_2 = e$. Now there exists a large bidisk 
$B$ where $f_b$ is an H\'enon-like map in the sense of \cite{hl}, in particular
injective. 
So it has a unique measure of maximal entropy $\log d$ in $B$, and it is easy to
prove that 
this measure is actually $\mu = T^+(f_b)\wedge T^-(f_b)$, where $T^\pm(f_b)$
are the global currents constructed above.

So in this case most preimages of points in $\supp(\mu)$ actually escape
$\supp(\mu)$ and therefore do not contribute to the dynamics in the natural extension.
We see with this example that it is insufficient to prove a statement like: ``for  a
disk $\Delta\subset f^n(L)$, we have at least $(1-\e)\lambda_2^n$ contracting inverse branches'', 
since the remaining histories could have full measure in the natural extension.  
Our process of selecting distinguished inverse branches in 
Proposition \ref{prop:positive} will give a way of
identifying the ``good'' preimages. 
\end{exa}

In view of these examples, one may wonder  whether there exists a rational map with small topological degree such that the
number of preimages of points on the support of $\mu$ is a constant  strictly between 1 and $\la_2(f)$?

Likewise, is there an example where the number of preimages of points on the support of $\mu$ is essentially non-constant?  By ``essentially'' we mean that the
degree does not only vary on a set of zero measure.  Notice that this is not incompatible with ergodicity or mixing; indeed it is
easy to construct unilateral subshifts of finite type 
with this property: consider for instance the subshift on two symbols $0\mapsto
1$, $1\mapsto 0$ or 1. 
This situation can also occur for basic sets of Axiom A endomorphisms.

\section{The negative exponent}\label{sec:negative}

In this section we study the contraction properties along $T^-$ to 
estimate the  negative exponent.\footnote{Lyapunov exponents are well-defined
only when (H3) is satisfied: of course here we mean item {\em i.} of Theorem \ref{thm:complete}.}
The main lemma is the following. Recall that $T^+\wedge \tau$ denotes the
transverse measure induced by $T^+$ on $\tau$. 

\begin{lem}[\`a la Lyubich]
 Let $\el = \set{D_t, \ t\in \tau}$ be a flow box subordinate to $T^+$. 
Then for every $\e>0$, there exists a positive constant $C(\e)$ and a transversal
$\tau(\e)\subset \tau$ such that 
$\m(T^+\wedge\tau(\e))\geq(1-\e)\m(T^+\wedge\tau)$ and
$$\forall n\geq 1, \ \forall t\in \tau(\e) ,\ \area(f^n(D_t))\leq \frac{C(\e)n^2
\lambda_2^n}{\lambda_1^n}$$
\end{lem}

\begin{proof} The method is  similar to \cite{birat}, though it requires 
substantial adaptation. We freely use the 
structure properties of strongly approximable laminar currents; 
the reader is referred to \cite{birat, structure} for  more details.

 It is enough to prove that for every $\alpha$ and every fixed $n$,
the transverse measure of the set of disks $D_t$ such that
$\area(f^n(D_t))>\alpha$ is 
smaller than $\frac{\lambda_2^n}{\alpha \lambda_1^n}$. Indeed if this is the
case, then  for every  $n\geq 1$
and every $c>0$, 
$$
(T^+\wedge\tau)\left(\set{t, \ \area(f^n(D_t)) >  \frac{c n^2
\lambda_2^n}{\lambda_1^n} }\right) < \frac{1}{cn^2},
$$
and it will suffice to sum over all integers $n$ and adjust the constant $c$ to
get the conclusion of the lemma.

\smallskip

We first need to analyze the action of $f$ on the transverse measure. This is a
local problem. 
Recall that  $T^+$ gives no mass to the critical set \cite{part1}, 
and consider an  open set $U$ such that 
 $f: U\cv f(U)$ is a biholomorphism. If $\tau$ is a transversal to some flow box
$\el$ contained 
in $U$, then $f(\tau)$ is a transversal to $f(\el)$ and from  the invariance
relation $f^*T^+= \lambda_1 T^+$ we infer that 
 $\m(T^+\wedge f(\tau))  = \lambda_1 \m(T^+\wedge \tau)$. 

\smallskip

Consider our original flow box $\el$ and fix $n$ and $\alpha$. Let
$\tau_\alpha\subset \tau$ be the set of disks such that 
$\area f^n(D_t)> \alpha$, and $\el_\alpha$ be the corresponding flow box. 
Sliding $\tau$ along the lamination and discarding a set of transverse measure
zero if necessary, we can arrange  that 
\begin{itemize}
\itm $\tau$ is contained in a holomorphic disk transverse to $\el$;
\itm $\tau\cap I(f^n)=\emptyset$ and $\tau\cap C(f^n)$ 
is a finite set;
\itm for every $p\in \tau_\alpha$, $p$ is the unique preimage of $f^n(p)$ in
$\tau_\alpha$. 
\end{itemize}
We want to estimate the transverse measure of $\tau_\alpha$.
 We exhaust $\tau_\alpha$ by compact subsets $\tau'_\alpha$, such that
$\tau'_\alpha \cap C(f^n) = \emptyset$, and 
for simplicity rename $\tau'_\alpha$ into $\tau_\alpha$. 
Observe that
$f^n(\el_\alpha)$ is contained in $T^+$ so its total mass relative to $T^+$, 
$\m(T^+ \wedge f^n(\el_\alpha))$ is bounded by  $1$. 
On the other hand we will prove that 
$$
\m(T^+ \wedge f^n(\el_\alpha))\geq
\frac{\alpha\lambda_1^n}{\lambda_2^n} \m(T^+\wedge \tau_\alpha),
$$
hence giving the desired result.
In constrast to the birational case ($\lambda_2= 1$), 
we have to take into account the fact that $f^n(\el_\alpha)$ will overlap
itself, in a way controlled by the topological degree 
$\lambda_2^n$. 

To give the idea for the rest of the proof, we first consider a model
situation: 
imagine a single disk $D$ which admits a partition into $\lambda_1^n$ 
pieces $D_i$ such that $f^n(D_i)$ has area greater than 
$\alpha$. Then $f^n(D)= \bigcup f^n(D_i)$ has area greater than $\alpha
\lambda^n_1/\lambda_2^n$ because 
$$
\forall p \in f^n(D), \ \#\set{i,\ p\in f^n(D_i)}\leq \lambda_2^n.
$$ 
The following computation is a ``foliated'' version of this counting argument. 

\smallskip

If $p\in \el\setminus C(f^n)$, $f^n$ is a biholomorphism in some neighborhood
$N$ of $p$. If $D_p$ denotes the disk of $\el$ through $p$, 
which is subordinate to $T^+$, then $f^n(D_p\cap N)$ is a disk subordinate to
$T^+$. Recall from \cite{structure} that   disks subordinate to 
$T^+$ do not intersect (i.e. they overlap when they intersect), 
so there is an unambiguous notion of leaf subordinate to $T^+$ (union of
overlapping subordinate disks),
and a disk subordinate to $T^+$ is contained in a unique leaf. Notice that for
every $t\in \tau_\alpha$, $D_t\cap C(f^n)$ is an at most countable number of
points. Since $f^n(D_t)$ has area greater than $\alpha$, we can remove a
small neighborhood $N$ of $C(f^n)$ such that
for every $t$, 
$\area(f^n(D_t\setminus N))>\alpha$. We further assume that $N\cap \tau_\alpha =
\emptyset$. 
Now if $p\in f^n(\el\setminus N)$, locally there
is a unique disk through $p$, subordinate to $T^+$ (namely $f^n(D_q\cap N(q))$
where $q\in \el\setminus N$, is any preimage of $p$ and $N(q)$ is a small
neighborhood of $q$), and we conclude that $f^n(\el\setminus N)$ is a piece of
lamination subordinate to $T^+$. 
If $p\in f^n(\el\setminus N)$, we denote by $L_p$ the leaf subordinate to $T^+$
through $p$. 

Consider the closed set $f^n(\tau_\alpha)$. This is locally a transversal to 
the lamination $f^n(\el\setminus N)$, but globally it can  
intersect a leaf many times. Its total transverse mass is 
$\lambda_1^n \m(T^-\wedge \tau_\alpha)$. Recall that every $p\in
f^n(\tau_\alpha)$ admits a unique preimage 
$q\in \tau_\alpha$, and $f^n(D_q\setminus N)$ is a piece of a holomorphic curve
through $p$ of area greater than $\alpha$, that 
we will denote by $\Delta_p$. By construction, a point in $f^n(\el\setminus N)$
belongs to at most $\lambda_2^n$ such $\Delta$'s.

Let $dA_L(x)$ denote area measure along the leaf $L$.
We have 
\begin{eqnarray} \label{eq:transverse}
 \alpha\lambda_1^n \m(T^+\wedge \tau_\alpha) &=&  \alpha \m(T^+\wedge
f^n(\tau_\alpha))\\ \notag
&\leq &\int_{f^n(\tau_\alpha)} \area(\Delta_p)
 d(T^+\wedge f^n(\tau_\alpha))(p) \\ \notag &=&
\int_{f^n(\tau_\alpha)}\left(\int_{L_p} \mathbf{1}_{\Delta_p}(x) dA_{L_p}(x)
\right) 
d(T^+\wedge f^n(\tau_\alpha))(p). 
\end{eqnarray}

Now, take a partition of $f^n(\el\setminus N)$ into finitely many flow boxes,
and in each flow box, project 
$f^n(\tau_\alpha)$ on a reference transversal $\tau_{\mathrm{ref}}$. 
For simplicity, we will assume that there is only one such flow box.  The
general case follows easily. 
We can decompose $f^n(\tau_\alpha)$, up to a set of zero transverse measure,
into at most countably many subsets $\tau_i$ 
 intersecting each leaf in a single point in the flow box, so that we get an
injective map 
$h_i: \tau_i\cv \tau_{\mathrm{ref}}$. Since the transverse measure is by
definition invariant under holonomy,
$(h_i)_*(T^+\wedge \tau_i) = (T^+\wedge\tau_{\mathrm{ref}})\rest{h_i(\tau_i)}$.
We can thus resume computation 
\eqref{eq:transverse} as follows:
\begin{align*}
 \int_{f^n(\tau_\alpha)}\int_{L_p} \mathbf{1}_{\Delta_p}(x) dA_{L_p}(x) 
d(T^+\wedge &f^n(\tau_\alpha))(p) = 
\sum_i \int_{\tau_i} \int_{L_p} \mathbf{1}_{\Delta_p}(x) dA_{L_p}(x)d(T^+\wedge
\tau_i)(p) \\
&= \int_{\tau_{\mathrm{ref}}}\int_{L_q}\sum_i \mathbf{1}_{\Delta_{h_i^{-1}(q)}}
(x)dA_{L_q}(x)d(T^+\wedge \tau_{\mathrm{ref}})(q)\\
&\leq \int_{\tau_{\mathrm{ref}}}\int_{L_q} \lambda^n_2
\mathbf{1}_{\bigcup_i\Delta_{h_i^{-1}(q)}}
(x)dA_{L_q}(x)d(T^+\wedge \tau_{\mathrm{ref}})(q) \\
&=\lambda^n_2 \int_{\tau_{\mathrm{ref}}}
\area\left(\bigcup_i\Delta_{h_i^{-1}(q)}\right) d(T^+\wedge
\tau_{\mathrm{ref}})(q) \\
&=\lambda^n_2 \m_{T^+}(f^n(\el\setminus N)) \leq \lambda^n_2,
\end{align*}
where the inequality on the third line comes from the fact that a given point
belongs to at most $\lambda_2^n$ disks $\Delta$.
We conclude that $\m(T^+\wedge \tau_\alpha) \leq \lambda^n_2/\alpha\lambda_1^n$
which was the desired estimate. 
\end{proof}

\begin{proof}[Proof of item {i.} in Theorem \ref{thm:complete}]
The conclusion now follows directly from \cite[p.236]{birat}. 
Here we give a simpler argument that avoids using the Birkhoff Ergodic Theorem.  This shows that taking sets of density 1 in \cite{birat} was superflous.

Since $\mu$ is the geometric intersection of $T^+$ and $T^-$, we can fix a finite disjoint union $A=\el_1\cup\cdots \cup\el_N$ of flow boxes $\el_i$ for $T^+$ such that
$\mu(\el_1\cup\cdots \cup\el_N)\geq 1-{\e}$.  For each $p\in \el_i$ let $D_p$ be the disk of $\el_i$ through $p$, and let $e^s(p)$ be the unit tangent vector to $D_p$ at $p$. 
By the previous lemma, we may discard from $A$ a set of plaques of small transverse measure (hence of small $\mu$ measure)
 to arrange that $\area(f^n(D_p))\leq \frac{c(\e)
n^2\lambda_2^n}{\lambda_1^n}$ for every $p\in A$ and {\em every} $n\geq 1$.  By slightly reducing the disks (hence losing one more $\e$ of mass) and applying the 
Briend-Duval area-diameter estimate \cite{briend-duval}, we further arrange that the diameter of $f^n(D_p)$ is controlled by $C(\e)n\lambda_2^{n/2}/{\lambda_1^{n/2}}$. 
Here, the constant depends on the geometry of the disks $D_p$, hence ultimately on $\e$ since the disks in $\el_1\cup\cdots \cup\el_N$ have bounded geometry.  So for $n$
large enough, $f^n(D_p)$ is contained in a single coordinate chart of $X$, and we infer that the derivative $df^n\left(e^s(p)\right)$ has norm controlled by $C(\e)
n\lambda_2^{n/2}/{\lambda_1^{n/2}}$. Thus
\begin{equation}\label{eq:lyap}
 \limsup_{n\cv \infty} \unsur{n}\log\abs{df^n\left(e^s(p)\right)}\leq
-\frac{\log({\lambda_1}/{\lambda_2})}{2}
\end{equation}
for all $p\in A$.  Letting $\e\to 0$, we obtain the same inequality holds for $\mu$-almost every $p$.
\end{proof}

\section{Distinguished inverse branches}\label{sec:positive}

In this section we study the dynamics along the current $T^-$.  This will lead to the estimate on the positive exponent and lay the groundwork for computing
entropy and proving local product structure.  
Our goal is to prove the following assertion about backward iteration on which nearly all subsequent results depend.

\begin{pro}\label{prop:positive}
For any $\e> 0$, there exists a family $\qq$ of disjoint cubes, a current $T^-_\qq\leq T^-$ (resp. $T^+_\qq\leq T^+$) marked and uniformly woven (resp. uniformly laminar) in each $Q\in\qq$, and a set of distinguished histories  such that
\begin{enumerate}[i.]
\item $\m(T^+_\qq\geom T^-_\qq) \geq 1 - \e$;
\item the set of distinguished histories $(x_j)$ with $x_0\in \supp T^+_\qq\geom T^-_\qq$ has measure $\geq 1-\e$ in the natural extension $\hat \mu$ of $\mu$;
\item if $f_{-n}$ is a distinguished inverse branch of $f^n$ along a disk $D$ in the web supporting $T^-_\qq$, then the
derivative of $f_{-n}$ is controlled by  $C(\e)n/\lambda_1^{n/2}$;
\item for every disk $D$ in the web supporting $T^-_\qq$, there exists at least one compatible sequence of distinguished inverse branches along $D$.
\end{enumerate}
\end{pro}

The proof of the bound in {\it iii.} originates in \cite{birat} and \cite{dt-saddle}, but to allow
for further applications, we undertake a more elaborate analysis of $\mu$.  The exact meaning of the word ``distinguished'' appearing throughout the statement will be made clear during the proof.

\begin{proof}[Proof of Proposition \ref{prop:positive}]
For a generic hyperplane section $L$, $S_k := \lambda_1^{-k}[f^k(L)]$ converges to $T^-$. 
As described earlier, we choose a generic subdivision $\widetilde\qq$ by cubes and extract from  
$S^-_k:= \lambda_1^{-k}[f^k(L)]$ a (uniquely marked) uniformly woven current $S^-_{k,\widetilde\qq}$ 
which is the restriction of $S_k$ to disks (graphs over one projection in $\qq$) of area 
not greater than $1/2$. 
%

Let $\e$ be a small positive number. Fix the subdivision $\widetilde{\qq}$ as
above,  together with a corresponding family $\qq$ of slightly smaller concentric cubes, 
homothetic to those of $\widetilde{\qq}$ by a factor $(1-\delta)$.  Then for an appropriate choice of  $\widetilde\qq$, 
\begin{equation}\label{eq:bdry} 
\mu(\widetilde{\qq}\setminus \qq)\leq 2(1-(1-\delta)^4).
\end{equation}
See \cite[Lemma 4.5]{isect} for a proof.

Now remove from $S^-_{k,\widetilde\qq}$ all components not intersecting $\qq$, and denote by $S^-_{k,{\qq}}$ the remaining current. Observe that $S^-_{k,\widetilde{\qq}}\rest{\qq} = S^-_{k,{\qq}}\rest{\qq}$.
As expained in \S \ref{subs:markings}, the currents $S^-_{k,\qq}$ are marked by measures on $Z(\widetilde{\qq}, 1/2)$ 
whose masses are bounded uniformly in $k$.

For reasons that will become clear in Theorem \ref{thm:tautological}
 below (roughly speaking, to get some invariance for the web supporting $T^-$),
we do some further averaging, by setting $T^-_k = \frac2k \sum_{i=\lfloor k/2\rfloor}^k S^-_i$, where $\lfloor \cdot \rfloor$
denotes the integer part function.
Observe that $T^-_k \cv T^-$ as $k\cv\infty$, and that 
restricting $T^-_k$ to the good components (as defined in \S\ref{subs:markings}) 
of $T_k^-$, we get
the uniformly woven current $T^-_{k,\qq} = \frac2k \sum_{i=\lfloor 
k/2\rfloor }^k S^-_{i,\qq}$. 

Since the transverse measures of $T^-_{k,\qq}$ have uniformly bounded mass, we may extract a strongly converging subsequence of $T^-_{k,{\qq}}$.  The limiting current $T^-_{\qq}$ is automatically uniformly woven (in each cube). 
We similarly construct currents $T^+_{{\qq}}$ associated to backward images of a generic line, though in this case, the extra averaging step, while harmless, is not necessary.

Recall that the $T^\pm_{\qq}$ are actually defined in the slightly bigger subdivision 
$\widetilde{\qq}$.  If the size of the  cubes is small enough, then
\begin{equation}\label{eq:size}
\m\left(T^+\wedge T^- - (T^+_{{\qq}}\geom T^-_{{\qq}})\right) 
\end{equation} will be small. 
So in the following we fix $\widetilde{\qq}\supset \qq$ so that the sum of the
error terms in 
\eqref{eq:bdry} and \eqref{eq:size}, together with the loss of mass 
coming from a neighborhood of the  set of  points where the two projections are
not  transverse fibrations is less than $\e/10$. 
For technical reasons, we further require that  $\mu(\fr\qq)=0$.

\vskip.2cm

By Lemma \ref{lem:nonatomic}, we can define the measures $\mu_k= T^+\wedge T^-_k$ and $T^+\wedge T^-_{k,
\qq}$ for every $k\geq 0$.  The mass of $\mu_k$ tends  to 1 for cohomological reasons.  We claim that in fact  $\mu_k\cv\mu$ as $k\cv\infty$.  Indeed, the wedge product of uniformly woven currents is geometric,  so 
from Lemma \ref{cor:lsc} we have for large $k$ (depending only on $\qq$, hence on $\e$) that 
\begin{equation}\label{eq:old}
\m\left( (T^+ \wedge T^-_{k, \qq})\rest{\qq} \right)\geq 
\m\left( (T^+_\qq \wedge T^-_{k, \qq})\rest{\qq}\right)
= \m\left( (T^+_\qq \geom T^-_{k, \qq})\rest{\qq}\right)\geq 1-\frac{2\e}{10}, 
\end{equation}
More generally, if $k_j$ is any sequence such that $T^-_{k_j, \qq}$ strongly
converges to some $T_{\qq}$, then 
$\lim T^+_\qq \geom T^-_{k_j, \qq} \geq  T^+_\qq \geom T^-_{\qq}$. 
So any cluster value of $\mu_k$ is larger than $T^+_\qq \geom T^-_{\qq}$, whence $\mu_k\cv\mu$. 

Denote by $\mu_{k,\qq}$ the measure $T^+_\qq \wedge T^-_{k,\qq} =T^+_\qq \geom
T^-_{k,\qq }$.

\medskip

\noindent{\bf Backward contraction for  $\mu_{k,\qq}$.}
Fix an integer $n\geq 1$.  Then for generic $L$ and any $k\geq n$, we have that $f^n: f^{k-n}(L)\cv f^k(L)$ is 1-1 outside some finite set.  So for every disk $D\hookrightarrow f^k(L)$, $f^n$ admits a unique (``distinguished'') 
inverse branch $f_{-n}:D\to f^{k-n}(L)$.  Since the area of $f^k(L)$ is no
greater than $C\lambda_1^k$ for some $C>0$, we have
$$
\# \set{ \text{ plaques } D \text{ of }
S^-_{k,\qq}, \text{ s.t. }\mathrm{Area} (f_{-n}(D))\geq
\frac{An^2}{\lambda_1^n}} \leq  \frac{C\lambda_1^k}
{An^2},
$$ 
and thus 
$$
\# \set{ \text{ plaques } D \text{ of }
S^-_{k,\qq}, \text{ s.t. } \exists n\leq k, \mathrm{Area} (f_{-n}(D))\geq
\frac{An^2}{\lambda_1^n}} \leq  \sum_{n=1}^k \frac{C\lambda_1^k}
{An^2} \leq \frac{C'\lambda_1^k} {A}
$$
Discarding the plaques in the latter set from $S^-_{k,\qq}$, we get a new uniformly woven
current, that we denote by 
 $S^-_{k,\qq}(A)$. In terms of the transverse measure, we have removed at
most  ${C'\lambda_1^k}/{A}$ 
Dirac masses\footnote{It is important here that the transverse
measure is a sum of Dirac masses on disks, not on arbitrary  chains.}.  Hence 
\begin{equation}\label{eq:A}
\m\left(m(S^-_{k,\qq}) - m(S^-_{k,\qq}(A)) \right)\leq \frac{C}{A}.
\end{equation}

We then form the current $T^-_{k,\qq}(A) = \frac2k \sum_{i=\lfloor k/2 \rfloor 
}^k S^-_{i,\qq}(A)$
(which also satisfies \eqref{eq:A}).  Extracting a strongly convergenct subsequence, we get a 
current $T^-_{\qq}(A)\leq T^-_{\qq}$. This family of currents increases to 
$T^-_{\qq}$ as $A$ increases to infinity (technically we need to take a sequence
of $A$'s and a diagonal extraction), hence by geometric intersection 
$T^+_\qq \wedge T^-_\qq(A)$ increases to  $T^+_\qq \wedge T^-_\qq$. We fix $A$
so that $M(T^+_\qq \wedge T^-_\qq(A))\geq 1-\e/2$. Relabeling, we now use $T^-_\qq$
to denote $T^-_\qq(A)$ and $T^-_{k,\qq}$ to denote $T^-_{k,\qq}(A)$.  We will show that the conclusions of
Proposition \ref{prop:positive} hold for this current $T^-_\qq$ and the current $T^+_\qq$ constructed above. 

If $n$ is fixed,  then  as soon as $k/2\geq n$, the ``distinguished'' 
preimage of each plaque of $T^-_{k,\qq}$ under $f^n$ has small area by construction. 
By lower semicontinuity  (Lemma \ref{cor:lsc}), we get that for large $k$, $\m(\mu_{k,\qq}) \geq 
1-\e/2$. 

\medskip

From the area-diameter estimate of \cite{briend-duval} and the Cauchy estimates, 
we see that the modulus of the  derivative of $f_{-n}$ along the plaques of
$T^-_{k,\qq}$ is small  in $\qq$. To be more specific, if $Q\in \qq$ and $D$ is
a plaque of $T^-_{k,\qq}$ in $\widetilde\qq$, the modulus of the derivative of
$f_{-n}$ along $D\cap Q$
 will be  controlled by $C(\e)n/\lambda_1^{n/2}$. 
 Indeed, recall that $D$ has area $\leq 1/2$ so by a simple compactness
argument, the moduli of annuli surrounding the connected components of $D\cap Q$
in $D$ are 
 bounded from below by a constant depending only on the geometry of $Q$ and
$\widetilde{Q}$, hence ultimately on $\e$. The estimate on the derivative then
follows from the original estimate of Briend and Duval.

\medskip

\noindent{\bf Distinguished inverse branches.}
For each $Q\in \qq$ and each plaque $D$ of $T^-_Q$ in $Q$, $D$ is the Hausdorff
 limit of a sequence of disks $D_k$ with the 
property that for every $1\leq n \leq k/2$, $f^{n}$ admits a {\em natural}
inverse branch $f_{-n}$ 
over $D_k$, with a uniform control on the derivative along $D_k$, of the form 
\begin{equation}\label{eq:derivative}
\abs{d(f_{-n}\rest{D_k})}\leq C(\e)n/\lambda_1^{n/2}.
\end{equation}

As already explained (see the examples in \S \ref{sec:idea}) it will be
important to isolate a set of 
``meaningful'' histories in the natural extension. Here is the crucial
definition.

\begin{defi}\label{def:dist}
Let $D$ be a disk subordinate to $T^-_\qq$.  We call an inverse branch $f_{-n}$ of $f^n$ along $D$ \emph{distinguished}
if there exists a sequence of disks $D_k$ subordinate to $T^-_{k,\qq}$ converging to $D$ such that the natural branches
$f_{-n}\rest{D_k}$ converge normally to $f_{-n}\rest{D}$.

We say that a history $(x_{-j})_{0\leq j\leq n}$ of length $n$ is \emph{$\qq$-distinguished} if $x_0\in \supp(T_\qq^+\geom T^-_\qq)$ and $x_{-n} = f_{-n}(x_0)$ for some distinguished inverse branch $f_{-n}$.
An infinite history $(x_{-i})_{i\geq 0}$ is $\qq$-distinguished if all its subhistories of length $n$ are $\qq$-distinguished.
\end{defi}
%

When there is no danger of confusion we will omit the `$\qq$' in $\qq$-distinguished; more generally, ``distinguished'' will stand for ``$\qq$-distinguished for some $\qq$''.  Taking normal limits of  the $f_{-n}\rest{D_k}$, we see that $f^n$ admits distinguished inverse branches on every disk suboordinate to $T^-_\qq$.
By diagonal extraction, we further find that every $x_0\in \supp(T_\qq^+\geom T^-_\qq)$ admits a distinguished (full) history. Moreover, attached to every $\qq$-distinguished history of $x_0$, there is a disk $D\ni x_0$ subordinate to $T^-_\qq$ and a sequence of distinguished inverse branches $f_{-n}\rest{D}$, with $f_{-n}(x_0) = x_{-n}$, compatible in the sense that $f\circ f_{-n-1}=f_{-n}$.
 Since the estimate in \eqref{eq:derivative} is uniform in $k$ we have the estimate
$\abs{d(f_{-n}\rest{D})}\leq C(\e)n/\lambda_1^{n/2}$. This proves items
{\em iii.} and {\em iv.} of Proposition \ref{prop:positive}.

\medskip

It remains to show that distinguished histories are overwhelming in the natural extension of $\mu$. 
Let $\xcd_\qq \subset \hat{X}$ be the set of distinguished histories $(x_j)$ of points $x_0\in\supp(\mu_\qq)$; likewise, let $\hat{X}^{\mathrm{dist}}$ be the increasing union of $\xcd_\qq$ as the diameter of the cubes in $\qq$ goes to zero. 
We prove that $\hat{\mu}  (\xcd)=1$ by proving that $\hat{\mu}  (\xcd_\qq)\geq 1-\e$. 

Let $\xcd_{-N, \qq}\subset \hat X$ consist of histories of points $x_0\in\supp\mu$ that are distinguished up to time $-N$. This is a decreasing sequence of subsets of $\hat{X}$, and $\bigcap_{N\geq 1} \xcd_{-N, \qq} = \xcd_\qq$.  It is enough to prove that $\hat{\mu}  (\xcd_ {-N, \qq})\geq 1-\e$ for all $N\geq 1$.  Now, $\hat{\mu}  (\xcd_ {-N, \qq})= \hat{\mu} \big(\hat{f}^N(\xcd_ {-N, \qq})\big)$
 and by definition 
$\hat{f}^N(\xcd_ {-N, \qq})$ is the set of sequences $(x_n)$ such that $x_0$ is
a cluster value of a sequence 
$f_{-N}(x_N^{(k)})$, with $\supp(S_{k,\qq}^-)\ni x_N^{(k)} \cv x_N$. 

Recall that the measure $\mu_{k,\qq}$ has mass larger than $1-\e$,
and satisfies
\begin{equation}\label{eq:muqk}
\mu_{k,\qq}:= T^+_\qq\wedge T^-_{k,\qq} \leq T^+ \wedge  T^-_{k,\qq}= \frac2k
\sum_{i=\lfloor k/2\rfloor }^k T^+ \wedge S^-_{i,\qq} . 
\end{equation}
We come now to the crucial point.  Since there is a natural $f_{-N}$ on $f^k(L)$ for $k>2N$, we may consider the measure $(f_{-N})_* \mu_{k,\qq}$, which has mass larger than $1-\e$. From \eqref{eq:muqk} we infer that 
$(f_{-N})_* \mu_{k,\qq} \leq T^+\wedge T^-_{k-N} + \sigma_k$, where $\sigma_k$ is a signed measure of total mass $O(\frac{N}{k})$. Consider any cluster value of this sequence of measures as $k\cv\infty$ and
denote it by $(f_{-N})_* \mu_{\qq}$ (this notation is convenient but somewhat improper). Then 
$(f_{-N})_* \mu_{\qq}\leq \mu$ because $\mu_k\cv\mu$ and $\m((f_{-N})_* \mu_{\qq}) \geq 1-\e$. Now if $(x_n)\in \hat X$ is {\em any} sequence such that $x_0 \in \supp((f_{-N})_* \mu_{\qq})$, then by definition $(x_n)\in 
\big(\hat{f}^N(\xcd_ {-N, \qq})\big)$ (note that the tail $(x_j)_{j<-N}$ of a history in 
$\xcd_ {-N, \qq}$ is arbitrary). Hence 
$$\hat\mu\left( 
\hat{f}^N(\xcd_ {-N, \qq})\right) \geq 
\hat\mu\left( \hat{\pi}_0^{-1} \left(\supp((f_{-N})_* \mu_{\qq}\right) \right) 
= \mu \left(\supp((f_{-N})_* \mu_{\qq}) \right) \geq 1-\e.$$
This finishes the proof of the proposition.
 \end{proof}

\begin{rqe}\label{rmk:laminarity}
\begin{enumerate}[i.]
\item From the first part of the proof we could directly deduce the existence of
the positive exponent (expansion in forward time), in the spirit of 
\cite{dt-saddle}. However, to obtain more complete results, we first construct the  
tautological extension. 
\item The laminarity of $T^+$ is only used in the proof to get lower semicontinuity properties of the wedge products. So if, for instance, $T^+$ has continuous potential, we can drop the laminarity assumption and get the same conclusion.
Notice additionally in this case that 
the wedge product $T^+\wedge T^-$ is ``semi-geometric'' in the sense that it is approximated from
below by $T^+\wedge T^-_\qq$ (see the proof of \cite[Remarque  4.6]{isect} or 
\cite[Remark 5.3]{birat}).
This is the setting of \cite{dt-saddle}, and might be useful for further applications.
\end{enumerate}
\end{rqe}

\section{The tautological extension}\label{sec:tautological}

So far we have constructed a family of marked uniformly woven currents $T_\qq^-= \sum T_Q^-$, increasing to $T^-$, 
with the property that for any disk appearing in the markings, $f^n$ admits exponentially contracting inverse branches. We say that such disks are {\em subordinate} to $T^-$. 

If $D$ is a disk subordinate to $T^-$, we define the measure $T^+\geom D$ as the increasing limit of the measures $T^+_{\qq_i}\geom D$ for our choice of increasing  subdivisions $\qq_i$.  Since $T^+\wedge T^-$ is a geometric 
 intersection, 
 if $D$ is a generic (relative to the marking) disk subordinate to $T^-$, then 
 $T^-\wedge [D]$ is well defined and  $T^+\wedge  [D] = T^+\geom D$.

\medskip

We now construct the tautological extension $(\check{X}, \check\mu, \check f)$ of $(X,\mu,f)$. This is roughly speaking 
the ``smallest'' space\footnote{This space does not depend canonically on $(X,\mu, f)$. On the other hand it 
depends canonically on $\mu$ viewed as a geometric intersection of marked woven currents.} where the unstable leaves become separated.  It will be realized as a disjoint union of flow boxes, foliated\footnote{We make no attempt to connect the flow boxes, as this would certainly lead to unwelcome topological complications.} 
by the lifts of the disks subordinate to $T^-$ (which play the role of unstable manifolds).    The marking data give us a lift of $T^-$ to a ``laminar current'' on
$\check X$.  Hence we get a lift $\check\mu$ of $\mu$, whose conditionals
on unstable manifolds are well understood. This will be our main technical step towards the understanding the conditionals of $\hat\mu$ on unstable manifolds in the {\em natural} extension. As suggested by the referee, it might be possible to analyze the conditionals of $\hat\mu$ directly from $\mu$ but we don't know how to do it.

\begin{thm}\label{thm:tautological}
 There exists a locally precompact and separable  space $\check{X}$, which is a
countable union of compatible flow boxes, 
together with a Borel probability measure $\check{\mu}$ and
a measure preserving map $\check f$, with the following properties.
\begin{enumerate}[i.]
\item There exists a projection $\check\pi: (\check{X}, \check\mu) \cv (X,\mu)$
 semiconjugating $f$ and $\check f$, i.e. $\check\pi\circ \check f = f\circ
\check \pi$.
\item $\check{X}$ admits a measurable partition $\check D$ whose atoms 
 $\check D(\check x)$, $\check x\in\check X$, are mapped homeomorphically
by $\check\pi$ to onto disks subordinate to $T^-$.
\item The conditional measure of $\check\mu$ on almost any atom $\check D(\check x)$ is induced by the current $T^+$ as follows: 
it is equal up to normalization to $\big((\check\pi\rest{D(\check x)})^{-1}\big)_*
(T^+\geom D)$
where $D=\check\pi (\check D(\check x))$.
\end{enumerate}
\end{thm}

\begin{proof}
As before, let $\qq$ be one among a fixed increasing sequence $(\qq_i)$ of subdivisions
with $\mu(\fr\qq_i)=0$.  The current $T_\qq$ is marked by a measure on the disjoint union $\coprod_{\tilde Q\in \tilde \qq} Z(\tilde Q,1/2)$.  To simplify notation, we will omit the $\,\tilde{}\,$ and the $1/2$ in the sequel.  Recall from \S \ref{sec:geom} that points in the tautological extension $\check Z(Q)$ of $Z(Q)$ are pairs $(x,D)$ with $x\in D\in Z(Q)$, and that the projection $\check\pi:\check Z(Q) \to Z(Q)$ is given by $\check\pi(x,D) = x$.

The principle of the proof is quite simple.  Each $D \in \supp m(T_Q^-)$ admits a natural lift $\check D$
to $\check Z(Q)$, and $T^+$ induces a measure on $\check D$ according to the formula
$$
T^+\geom \check D := \big((\check\pi\rest{\check D})^{-1}\big)_*(T^+\geom [D]),
$$ 
Averaging with respect to the markings then gives a measure
$$
\check\mu_Q := \int (T^+\geom \check D) \ d(m(T_Q^-)) (D)
$$ 
that projects on $\mu_Q$. 
We call $\check\mu_Q$ the \textit{tautological extension} of $T^+\geom T_Q^-$. 
What remains is to make sense of the ``increasing limit'' $\check\mu$ of $\sum_{Q\in\qq_i} \mu_Q$ as $i\to\infty$.  In particular, we need to construct the space $\check X$ that carries $\check\mu$, and then show that $\check\mu$ is invariant under some naturally associated map $\check f$.
%

\medskip
 
To construct $\check X$ we recall from Lemma \ref{lem:chain} that for generic subdivisions, almost all
chains appearing in the markings $m(T^-_\qq)$ are disks transverse to the boundary (this is an open subset of chains).  Let $O_1\subset\supp(m(T_{\qq_1}^-))$ be the relatively open, full measure subset of boundary-transverse disks. Set
$E_1 = O_1$  and let $\check E_1$ be the tautological bundle over $E_1$. 

Now assume that for $1\leq j \leq i-1$ the sets $E_j$ and $\check E_j$ have been
constructed. Consider the marked current $T_{\qq_i}^-$, and as before denote by $O_i$ the open subset of $\supp(m(T_{\qq_i}^-))$ made of boundary-transverse disks. We add to $\bigcup_{j\leq i-1} E_j$ the smaller set $E_i \subset O_i$ consisting only of those disks which have not previously appeared: i.e.
$$
E_i = O_i \setminus \set{D : D\subset Z \in \supp(m(T_{\qq_j}^-))}.
$$ 
As is easily verified, $E_i$ is open in $O_i$ and therefore also in the locally compact metric space
$\coprod_{Q\in \qq_i} Z(Q)$. 
\medskip

Now take $\check X$ to be the disjoint union  $\check X = \coprod_i \check E_i$. 
By Lemma \ref{lem:graph}, there is no folding in $m(T_{\qq_j}^-)$ so each
 $\check E_i$ is laminated by the disks of $E_i$. We can endow $\check X$ with a
natural topology 
by putting the natural topology on each $\check E_j$, and declaring that each
$\check E_j$ is open and closed. 
This topology is even induced by a metric where each of the $\check E_i$ is
bounded, and 
at definite (positive) distance from the others.
This makes $\check X$ a locally precompact and separable space. 
We could also take its completion to get a locally compact and separable space
but we will not need it. 
Each  $E_i$ is naturally partitioned by the disks of $E_i$, giving rise to the
measurable partition $\check D$ of {\em ii.}, and 
can be covered by a countable family of compatible flow boxes.

To define the measure $\check\mu$, we note that if $D\subset D'$ with $D\in Z(\qq)$ (resp $D'\in Z(\qq')$), then there is a natural inclusion $\check D \hookrightarrow \check D'$. Hence we can view the tautological extension of $T^+\geom T_{\qq_i}^-$ as a measure $\check\mu_{\qq_i}$ supported on $E_1\cup\cdots \cup E_i$ (rather than $Z(\qq_i)$). This defines an increasing sequence of measures\footnote{Note that $\check\mu_{\qq_i}$ might exceed $\check\mu_{\qq_{i-i}}$ not only on $E_i$ but also on those earlier sets $E_j$ that contain disks in $O_i$.} in $\check X$.  The limit 
$\check\mu$ is a Borel probability measure.
The conditional of the measure $\check\mu_{\qq_i}$ on a disk $D$ of $\check E_j$, is
by definition induced by $T^+$ (and independent of $i\geq j$), so statement {\em iii.} follows.

\medskip

Finally, we seek to construct the measurable map $\check f$ projecting onto $f$ and
leaving $\check\mu$ invariant.  We want to define $\check f$ as follows: $\check f (x,D) = (x',D')$ if $f(x)=
x'$ and $\mathrm{germ}_{x'}(f(D)) = \mathrm{germ}_{x'}(D')$.  The details of this are a little involved, 
however.

\medskip

To begin with, let us recall some notation from Proposition \ref{prop:positive}.  In the course of proving this result, we introduced currents $T_k^- = \frac2k \sum_{j=k/2}^k \unsur{\lambda_1^j}S_j$ together with the restriction $T^-_{k,\qq}$ of $T_k^-$ to disks which were `good' relative to $\qq$.  We also had measures $\mu_k = T^+\wedge T^-_k$ and $\mu_{k,\qq}= T^+\wedge T^-_{k,\qq}$ (note the slight change in the definition of the latter).

It is immediate from the definitions that the difference $\sigma_k := f_*\mu_k - \mu_k$ has mass no greater than $4/k$.  Another useful observation is that, for purposes of comparing the various measures we have defined, we can be a little flexible concerning their domains.  Since $\mu_k$ gives no mass to points (Lemma \ref{lem:nonatomic}) and $\mu_{k,\qq}$ is concentrated on countably many disks of $Z(\qq)$ with discrete intersections, we can lift $\mu_{k,\qq}$ canonically to a measure $\check \mu_{k,\qq}$ on $\check Z(\qq)$. So we may regard $\mu_{k, \qq}$ as a measure on the cubes of $\qq$ or alternatively as a measure on the the tautological bundle $\check Z(\qq)$.  In a similar vein,  we may regard $\check\mu\rest{\check E_1\cup \cdots\cup\check E_i}$ as a measure on $\check Z(\qq_i)$ rather than $\check X$.  
\medskip

\begin{lem}\label{lem:cluster} 
 Let $\check\nu_{\qq_i}$ be any cluster value of the sequence of measures
$\check\mu_{k,\qq_i}$. 
Then $\check\nu_{\qq_i} - \check\mu_{\qq_i}$ is a signed measure with 
$$\m\left( \check\nu_{\qq_i} - \check\mu_{\qq_i} \right) =\e(\qq_i)$$ 
where $\e(\qq_i)$ depends only on $i$ and tends to zero as $i\cv\infty$.
\end{lem}

\begin{proof}[Proof of Lemma \ref{lem:cluster}] 
Since $T^+\geq T^+_{\qq_i}$, we have
$$
\check \mu_{k,\qq_i}\geq \int (T^+_{\qq_i}\geom \check D) \ d(m(T_{k,\qq_i}^-)) (D).
$$ 
Taking cluster values on both sides, and using  lower semicontinuity  (Lemma \ref{cor:lsc}) we infer that 
$$
\check\nu_{\qq_i} \geq \int (T^+_{\qq_i}\geom \check D) \ d(m(T_{\qq_i}^-))
(D).
$$
On the other hand,
$$
\check\mu_{\qq_i} = \int (T^+\geom \check D) \ d(m(T_{\qq_i}^-)) (D) \geq \int (T^+_{\qq_i}\geom \check D) \ d(m(T_{\qq_i}^-))(D).
$$
Hence $\check\mu_{\qq_i},\check\nu_{\qq_i}$ are both measures with at most unit mass, and both are bounded below by a measure which, by geometric intersection, has mass at least $1-\e(\qq_i)$. 
\end{proof}

Continuing with the proof of the theorem, we let $A_{k,\qq_i} = \supp T^-_{k,\qq_i} \cap f^{-1}(\supp T^-_{k,\qq_i})$. That is, $A_{k,\qq_i}$ consists of points that `go from large disks to large disks.' Since 
$f_*\mu_k = \mu_k + \sigma_k$ and $f$ is essentially 1-1 on $\supp T^-_{k}$, we
have that $A_{k,\qq_i}$ has almost full mass:
$$
\mu_k(A_{k,\qq_i}) \geq \mu_k(\supp T^-_{k,\qq_i})- \mu_k (\supp T^-_{k} - \supp T^-_{k,\qq_i}) - \frac4k.
$$
By geometric intersection, the second term on the right hand side is of the form $\e(k,\qq_i)$ with $\lim_{k\cv\infty}\e(k,\qq_i) = \e(\qq_i)$ and $\e(\qq_i)\cv 0$ as $i\cv\infty$.  So since $\mu_{k,\qq_i} = \mu_k|_{\supp T_{k,\qq_i}}$, we infer that
\begin{equation}\label{eq:muqui}
 \m\left(\mu_{k, \qq_i} - \mu_{k,\qq_i}\rest{A_{k,\qq_i}}\right) \leq
\m\left(\mu_{k} - \mu_{k,\qq_i}\rest{A_{k,\qq_i}}\right) \leq \e(k,\qq_i) +
\frac4k.
\end{equation}
In the same way we obtain that 
\begin{equation}\label{eq:fetoile}
\m\left(\mu_{k, \qq_i} - f_*\mu_{k,\qq_i}\rest{A_{k,\qq_i}}\right)\leq
\e(k,\qq_i) + O(\frac1k).
\end{equation}

\medskip

Lifting to the tautological bundle, we get a set $\check A_{k,\qq_i}$. Consider a cluster value $\nu_{A}$ of the sequence of measures $\check\mu_{k,\qq_i}\rest{\check A_{k,\qq_i}}$. From Lemma \ref{lem:cluster} and
\eqref{eq:muqui}, and since there is no loss of mass in the boundary, we get
that 
$$
\m\left( \nu_A - \check\mu_{\qq_i} \right)\leq \m\left( \nu_A -
\check\nu_{\qq_i} \right) + 
\m\left( \check\nu_{\qq_i} - \check\mu_{\qq_i} \right)
 =\e(\qq_i),
$$ 
where $\check\nu_{\qq_i}$ is a cluster value of the sequence of measures $\check\mu_{k,\qq_i}$.

Now set $\check A_{\qq_i} :=\limsup_{k\to\infty} A_{k,\qq_i}$.  Since this contains $\supp( \nu_A)$; we
infer that $\check\mu_{\qq_i}(\check A_{\qq_i})\geq 1-\e(\qq_i)$.  We claim moreover that there is a well-defined 
map $\check f_i:\check A_{\qq_i} \to Z(\qq_i)$ given by $f_i(x,D) = (f(x),D')$ where $D'\in Z(\qq_i)$ coincides with
$f(D)$ near $f(x)$.  To see that this works, observe that by definition of $\check A_{\qq_i}$, there is a sequence
of pointed disks $(x_k,D_k) \to (x,D)$ such that $D_k$ is subordinate to $T^+_{k,\qq_i}$ and $f(D_k)$ coincides near
$f(x_k)$ with some other disk $D'_k$ subordinate to $T^+_{k,\qq_i}$.  Taking $D'$ to be a cluster value of the
$(D'_k)_{k\in\N}$, we see that $\check f_i$ is indeed well-defined.
Lemma \ref{lem:cluster} and \eqref{eq:fetoile} tell us additionally that 
\begin{equation}\label{eq:fetoilebis}
 \m\left((\check f_i)_*(\check\mu_{\qq_i}\rest{\check A_{\qq_i}}) -
\check\mu_{\qq_i} \right) = \e(\qq_i).
\end{equation}

Rephrasing the preceding construction in terms of $\check X$, we have constructed a set $\check A_{\qq_i}\subset \check E_1\cup\cdots\cup\check E_i$, with $\check\mu$-mass $\geq 1-\e(\qq_i)$, together with a map $\check f_i:A_{\qq_i} \to \check E_1\cup\cdots\cup\check E_i$, that coincides with the action of $f$ on the space of germs.  If, when refining the subdivision, we are careful to extract our subsequences from those chosen for earlier subdivisons, then we
will obtain an increasing sequence of subsets $(\check A_{\qq_i})$, with the compatibility (say $\qq_j$ is finer than $\qq_i$) $\check f_j\rest{\check A_{\qq_i}} = \check f_i$. So the maps $\check f_i$ piece together to form a single map $\check f$ defined on a full measure subset of $\check X$.  Furthermore, since 
$$
\m\left ((\check f_i)_*\check\mu\rest{\check A_{\qq_i}} - \check\mu\rest{\check E_1\cup\cdots\cup\check E_i}  \right) = \e(\qq_i)
$$
we infer that $\check\mu$ is $\check f$-invariant.
\end{proof}

The following proposition clarifies the relationship between iteration on $\check X$
and the construction of the previous section.
 We say that a disk has size $\geq r$ if it belongs to
$Z(Q)$ for some	 cube $Q$  of size $\geq r$.
 
 \begin{pro}\label{lem:technical}
 For every fixed positive integer $\ell$, there exists a set $\check
A_\e(\ell)\subset \check X$ of $\check\mu$-measure $\geq 1-\e$ such that if 
$\check x = (x,D)\in\check  A_\e(\ell)$ then
 $\check f^\ell(\check x)$ has the following properties:
 \begin{enumerate}[i.]
 \item $f^\ell(D)$ coincides near $f^\ell(x)$ with a disk of size $\geq r(\e)>0$;
 \item $f^\ell:D\cv f^\ell(D)$ is univalent, 
 with derivative larger than $C(\e)\lambda_{1}^{\ell/2}\ell^{-1}$;
 \item the orbit segment $x, f(x), \ldots, f^\ell(x)$ is distinguished.
 \end{enumerate}
 \end{pro}

\begin{proof} Replace the set $A_{k,\qq_i}$ used in the previous proof with the analogous set $A_{k,\qq_i}(\ell)$ of points $x\in \supp T^-_{k,\qq_i}$ such that $f^\ell(x)\in\supp T^-_{k,\qq_i}$. Define $A_{\qq_i}(\ell) =
\limsup(A_{k,\qq_i}(\ell))$ and consider as before the tautological bundles over these sets.  As we now explain, it suffices to take $\check A_\e(\ell) = \check A_{\qq_i}(\ell)$ for large enough $i$.

We first need to check that $\check f^\ell (\check x)$ is well defined for almost every point in $\check A_{\qq_i}(\ell)$.  The point is that there is a piece of disk in $Z(\qq_i)$ sent by $f^\ell$ to
a piece of disk of $Z(\qq_i)$, however  along the branch $x, f(x), \ldots, f^\ell(x)$ the disk can become small. Nevertheless if $x\notin \mathrm{Crit}(f^\ell)$, $f^q$ is locally invertible at $x$ for $1\leq q\leq
\ell$, so  all\footnote{technically, we need to avoid the measure zero set of $x$ sent by $f^q$ onto boundaries of cubes} the germs $f^q(D)$ are traced on disks of \emph{some}, possibly much smaller, size  $\qq_{i'}$ depending on $\ell$.  The same holds for disks subordinate to $T^-_{k, \qq_i}$ that approximate $D$.  We conclude that 
 $\ell$ successive iterates of $\check f$ are defined at $\check x$.

Now, the conclusions of the lemma follow easily from the analysis of distinguished inverse branches in the proof of Proposition \ref{prop:positive}.  The only point that needs explanation is {\em ii.} If $(x,D)\in\check
A_{\qq_i}(\ell)$, then by construction, $(x,D)$ is the limit of a sequence of
pointed disks $(x_k,D_k)$ with $\norm{df^\ell_{x_k}(e(\check x_k))}\geq C(\qq_i)
\lambda_1^{\ell/2}\ell^{-1}$, where $e(\check x_k)$ is the unit tangent vector to
$D_k$ at $x_k$. Recall from Lemma \ref{lem:graph} that since the current $T^-$ is strongly approximable, there is no
multiplicity in the convergence of the disks subordinate to $T^-_{\qq,k}$.  So if $\check x_k\cv\check x$ in this construction, then $e(\check x_k)\cv e(\check x)$.
It follows for $x\notin I(f^\ell)$ that $df^\ell_{x_k}(e(\check x_k))\cv df^\ell_x(e(\check
x))$, giving the desired estimate.  We note that since the estimate extends across finite sets, it holds even at points in $I(f^\ell)$.
\end{proof}

Now we can estimate the positive exponent. For convenience here we take for granted that $\check\mu$  
is ergodic, a fact we will prove in Corollary \ref{cor:ergodic} below.

\begin{cor} For $\mu$-a.e.  $x$ there exist a tangent vector $e^u$  at $x$
and a set of integers $\nn'\subset \nn$ of density 1 
 such that 
\begin{equation}\label{eq:eubis}
 \liminf_{ \nn'\ni n\cv\infty} \unsur{n} \log\abs{df^n(e^u(x))}\geq 
\frac{\log\lambda_1}{2}.
\end{equation}
\end{cor}

Given $\check{x} = (x,D) \in \check X$, we let $e(\check{x})$ denote the tangent vector to $D$ at $x$. 
The proof makes evident that one can take $e_u = e(\check x)$ for $\check{\mu}$ 
a.e. $\check{x}$.

\begin{proof} We have the following lemma from elementary measure theory (see
below for the proof).
\begin{lem}\label{lem:elementary}
 Let $(Y,m)$ be a probability space, and $(A_n)_{n\geq 1}$ a collection of sets
of measure $\geq 1-\e$. Then for every $\delta>\e$, 
$$m\left(\set{y\in Y, y\in A_n \text{ for a set of integers }n\text{ of density
}\geq  1-\delta}\right)\geq 1-\frac{\e}{\delta}.$$
\end{lem}

With notation as in Proposition  \ref{lem:technical}, let 
$$\check A_\e = \set{\check x,\ \check x \in \check A_\e(\ell)  \text{ for a set
of integers }\ell\text{ of density }\geq  1-\sqrt{\e}}.$$
By the previous lemma, $\check\mu(\check A_\e)\geq 1-\sqrt{\e}$.  
If $\check x\in \check A_\e$, then \eqref{eq:eubis} holds for a set of integers
$\nn_\e$ of density $\geq 1-\sqrt{\e}$.
To conclude, observe that  $\check A_\e$ is an invariant set, so that by
ergodicity  it has full measure.
\end{proof}

\begin{proof}[Proof of Lemma \ref{lem:elementary}] 
Consider the function $\varphi_N = \sum_{n=1}^N \mathbf{1}_{A_n}$ (with possibly
$N=\infty$). 
We have that $0\leq \varphi_N\leq N$ and $\int\varphi_N dm \geq (1-\e)N$. We
leave the reader prove that for every $\delta>\e$, 
$$m(B_N) \geq 1-\frac{\e}{\delta} \text{ where } B_N =\set{y,\
\varphi_N(y)\geq(1-\delta)N}.$$ 
In particular taking $\delta$ close to 1, at this point we conclude that for
every fixed $C$ and $N$ large enough, 
$m(\set{\varphi_N\geq C}) \geq 1-\e$. In particular for every $C$
$m(\set{\varphi_\infty \geq C}) \geq 1-\e$, so the set of points belonging to
infinitely many $A_n$ has measure $\geq 1-\e$.   

Now the $(B_N)$ themselves form an infinite collection of sets of measure $\geq
1-{\e}/{\delta}$ so applying the same reasoning proves that the set of $y$
belonging to infinitely many $B_N$'s has measure $\geq  1-{\e}/{\delta}$ which
is the desired statement.
\end{proof}

\section{The natural extension and  entropy}\label{sec:entropy}

In this section we analyze the natural extension of $(X,\mu,f)$. There are two
main steps: prove that the natural extension of $(X,\mu,f)$ is the same as that of $(\check X,\check\mu,\check f)$, and 
analyze the conditional measures of $\hat \mu$ relative to the unstable partition in the natural extension.

We first show that different disks subordinate to $T^-$ correspond to 
different histories, as one would expect for unstable manifolds. 

\begin{pro} \label{lem:separ}
 Let $\hat{x}\in \hat{X}$ be a $\hat\mu$-generic point. Then there exists a
{\em unique}   disk $D$ subordinate to some $T^-_\qq$, together with a
 sequence of inverse branches $f_{\hat{x}, -n}$ defined on 
$D$,  such that $f_{\hat{x}, -n}(x_0)=x_{-n}$  and $f_{\hat{x}, -n}$ 
contracts exponentially on $D$.
\end{pro}

\begin{proof}
Proposition \ref{prop:positive} tells us that distinguished histories have full measure in the natural extension.
So we may assume that $\hat{x}$ is $\qq$-distinguished for some $\qq$.  We therefore have a disk $D\ni x_0$ that is subordinate to $T^-_\qq$ and equipped with a compatible sequence $(f_{-n})$ of inverse branches such that $f_{-n}(x_0) = x_{-n}$.  Proposition \ref{prop:positive} further guarantees that $f_{-n}$ is exponentially contracting on $D$. It remains to show $D$ is unique. 

Consider the measure $\mu_\qq = T^+_\qq\geom T^-_\qq$.  By construction, for each $x\in \supp(\mu_\qq)$ there exists a radius $r = r(x)$ such that  
the disks subordinate to $T_\qq^+$ are  submanifolds in $B(x,r)$ and get contracted at uniform exponential speed $O(n\la_2^n/\lambda_1^n)$ under forward iteration. Every point in $\supp(\mu_\qq)$ has such a local ``stable manifold'',
which is then unique because $T^+$ is strongly approximable and laminar \cite[Theorem 1.1]{structure}. 
Let $\el^s\cap B(x,r)$ be the stable lamination near $x\in \supp\mu_\qq$, that is, the union of stable manifolds of points in $\supp(\mu_\qq)\cap B(x,r)$.
Since the image of a disk subordinate to $T^+$ under $f$ (resp. under a branch of $f^{-1}$ defined in an open set) 
is a disk subordinate to $T^+$, the stable lamination is invariant under the dynamics: 
$f(\el^s\cap B(x,r)) \subset \el^s$.

Now suppose for some generic history $\hat{x}$, that $D_1,D_2$ are two disks through $x_0$ satisfying the conclusions
of the proposition.  Then for $n$ large enough, we have $f_{-n}(D_j)\subset B(x_{-n},r_0/2)$ where $r_0 = r(x_0)$.  And since $x_0 \in \supp \mu_\qq \subset\supp\mu$, Poincar\'e recurrence for $\hat\mu$ 
 gives an infinite set $S\subset \N$ of $n$ such that $x_{-n} \subset B(x_0,r_0/2)$.  So if $L$ is any leaf in $\el^s \cap B(x_0,r)$ that meets both $D_1$ and $D_2$, then for every $n\in S$, $f_{-n}(L)\cap B(x_0,r_0)$ is contained in 
 a leaf $L'$ intersecting the preimages of either disk $f_{-n}(D_j)$.  Contraction of stable leaves then gives.
$$
\mathrm{distance}(L\cap D_1,L\cap D_2) \leq Cn\left(\frac{\lambda_2}{\lambda_1}\right)^n \mathrm{diameter}(L')  = O \left(\frac{n\lambda_2^n}{\lambda^n_1}\right).
$$
As this is true for all $n\in S$, we conclude that $L\cap D_1 = L\cap D_2$.  Similarly, neither $D_j$
can be entirely contained in the stable leaf through $x_0$, because pulling back and pushing forward would, in the same fashion, show that the diameter of $D_j$ vanishes.

Now we can conclude.  Leaves in $\el^s\cap B(x_0,r_0)$ accumulate on the one through $x_0$.  Hence there are infinitely many such leaves intersecting both $D_1$ and $D_2$.  Since the intersection points all lie in $D_1\cap D_2$, we see that $x_0$ is an accumulation point of $D_1\cap D_2$.  It follows that $D_1 = D_2$; i.e. the disk $D$ in the proposition is unique.
\end{proof}

\begin{rqe}\label{rmk:separ}
As is clear from the proof, the uniqueness assertion of the proposition holds for a given distinguished history $\hat x$ as long as there are infinitely many $n$ for which $x_{-n}\in \supp(\mu_\qq)$, and we only need $f_{\hat x,-n}$ to be contracting for these $n$. 

The proof also makes clear that the disk $D$ is unique \emph{regardless of whether it is subordinate to $T_\qq^-$}.  This shows that the web supporting $T^-_\qq$ is essentially independent, along histories in $\hat X$, of the manner in which it was constructed.  It also shows that for almost any $\hat x\in\hat X$, the disk $D$ is the only reasonable candidate for a `local unstable manifold' associated to $\hat x$.  We will therefore refer to such disks as local unstable manifolds and to the resulting partitions of $\hat\mu$ and $\check\mu$ as the `unstable partition' of each.  We should emphasize, however, that we do not know whether $D$ is the full local unstable set of $\hat x$---i.e. whether something like the local unstable manifold theorem holds in the present context.
\end{rqe}

\begin{pro}\label{prop:hatcheck}
The natural extension $(\hatcheck{X},\hatcheck{\mu} ,\hatcheck{f})$ of $(\check X, \check\mu , \check f)$ is measurably
isomorphic to that of $(X,\mu,f)$.
\end{pro}

\begin{proof}
The natural extension of $(\check X, \check\mu , \check f)$ is the set of sequences $(\check x_n) = (x_n,D_n)$, indexed by $\zz$, with $\check f(\check x_n)= \check x_{n+1}$.  Observe that if $x_0$ does not belong to $\bigcup_{n\geq 0}
f^n(C(f^n))$, then for positive $n$, there is a unique germ $D_{-n}$ at $x_{-n}$ such that $f^n(D_{-n})
= D_0$. In particular the whole sequence $(D_n)$ is determined by 
$D_0$ and $(x_n)$. 

Furthermore, since $(\hatcheck{X},\hatcheck{\mu} ,\hatcheck{f})$
is an invertible dynamical system projecting onto $(X,\mu,f)$, the universal property of $(\hat X,\hat\mu, \hat f)$
gives us an intermediate semiconjugacy 
$\eta : (\hatcheck{ X},\hatcheck{\mu} ,\hatcheck{f})\cv (\hat X,\hat\mu, \hat
f)$. 
In explicit terms, $\eta(x_n, D_n) = (x_n)$. 

Since the set of distinguished histories $\xcd$ has full measure in $\hat X$, we get that for $\hatcheck{\mu}$-a.e. $(x_n,D_n)$, $(x_n)\in \xcd$. By Proposition \ref{lem:separ} above, associated to the sequence $\hat x  = (x_n) \in \xcd$, there is a unique germ of disk $D(\hat x)$ through $x_0$ which is contracted exponentially in the past along the branch $x_{-n}$. The proof will be finished if we show that for  $\hatcheck{\mu}$-a.e. $(x_n,D_n)$, $D_0 = D(\hat x)$. 
Indeed, since $(D_{n})$ depends only on $D_0$, we will have found an inverse for $\eta$. 

Consider the sets $\check A_\e(\ell)$ as defined in Proposition \ref{lem:technical},
and define $\hatcheck{A}_\e(\ell)$  to be the set of sequences 
$(\check{x}_n)\in \hatcheck{X}$ with $\check{x}_{-\ell}\in \check A_\e(\ell)$.
We have that  $\hatcheck{\mu}(\hatcheck{A}_\e(\ell))\geq 1-\e$. Hence
by lemma \ref{lem:elementary} there is a set
$\hatcheck{A}_\e\subset\hatcheck{X}$ of measure $\geq 1-\sqrt{\e}$ of points belonging 
to $\hatcheck{A}_\e(\ell)$ for infinitely many $\ell$. By definition, if $(x_n,D_n)\in \hatcheck{A}_\e$, then
for infinitely many integers $\ell$, the (germ of) 
disk $D_0$ is contracted by the branch of $f^{-\ell}$ sending $D_0$ to
$D_{-\ell}$. By Proposition \ref{lem:separ} and 
Remark \ref{rmk:separ} we conclude that $D_0 = D(\hat x)$. 
\end{proof}

\begin{cor}\label{cor:ergodic}
 The measure $\check \mu$ is ergodic under $\check f$, and we have equality
between entropies $h(f,\mu) = h(\check f,\check\mu)$.
\end{cor} 

\begin{proof} 
Since $\mu$ is ergodic, 
so is $\hat \mu$, and therefore $\check\mu$. 
Also we have that  $h(f,\mu) \leq h(\check f,\check\mu)\leq h(\hat f,\hat \mu) =
h(f,\mu)$.
\end{proof}

From now on, depending on the context, we can think of  $(\hat X, \hat
\mu, \hat f)$ as the natural extension of either $(X,\mu,f)$ or $(\check
X,\check \mu, \check f)$.  
Figure \ref{figure:tautological} illustrates the relationship between the natural and tautological extensions. 
A disk subordinate to $T^-$ and its preimages under $\check\pi$ and $\pi_0$ has been underlined.
The notation $\check \el^s$, $\check\el^u$ is introduced in the proof of Theorem \ref{thm:product}
The existence of the
dashed arrow is ensured by the previous proposition, i.e. by the fact that
$\eta$ is a measurable  isomorphism. With notation as above, it is defined by 
$ \hat x \mapsto (x_0, D(\hat x))$.

\begin{figure}[h]
\begin{center}
\psfrag{u}{$W^u$}
\psfrag{s}{$W^s$}
\psfrag{T+}{$T^+$}
\psfrag{T-}{$T^-$}
\psfrag{Ducheck}{$\check\el^u$}
\psfrag{Dscheck}{$\check \el^s$}
\psfrag{X}{$X$}
\psfrag{Xcheck}{$\check X$}
\psfrag{Xhat}{$\hat X$}
\psfrag{pizero}{$\pi_0$}
\psfrag{picheck}{$\check\pi$}
\psfrag{eta}{}
\includegraphics[scale=0.5]{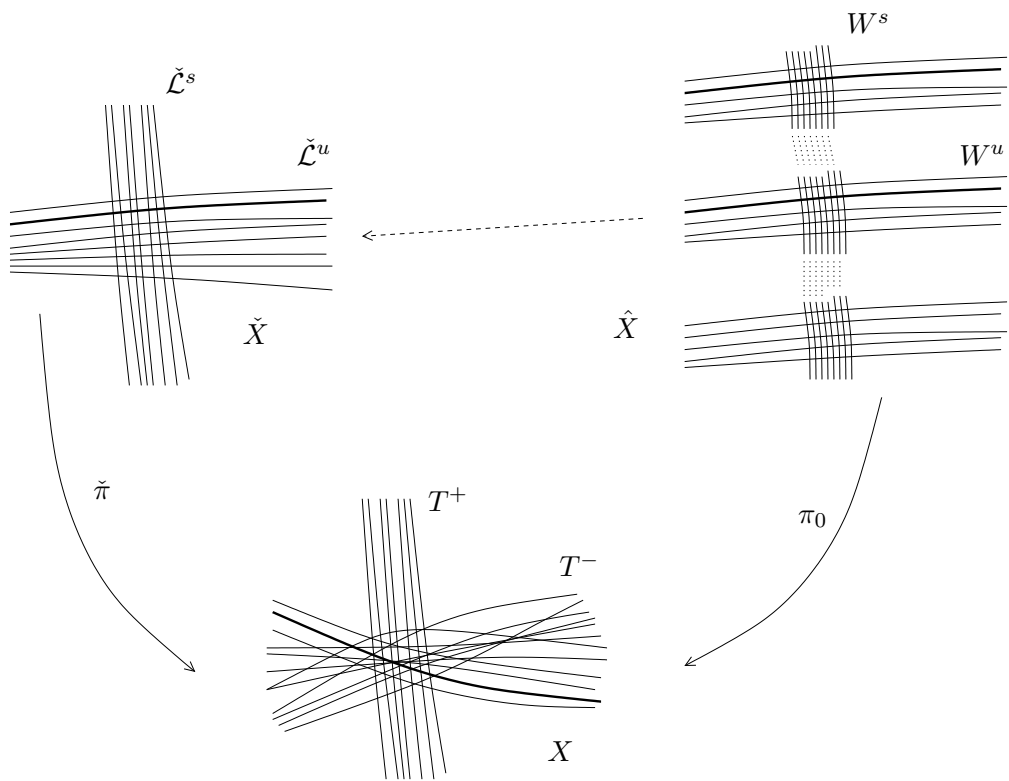}
\caption{Schematic picture of the tautological and natural extensions}\label{figure:tautological}
\end{center}
\end{figure}

\medskip
\begin{center}
 $\diamond$
\end{center}
\medskip

\noindent{\bf Aside: unstable manifolds and the natural extension.} 
As it is well known, defining unstable manifolds for a non invertible system requires working in the natural extension. 
The unstable set of $\hat x\in\hat X$ is 
$
W^u(\hat x ) = \set{ \hat y \in \hat X, \lim_{n\cv\infty} \mathrm{dist}(\hat
f^{-n}(x), \hat f^{-n}(y)) =0}.
$
Under some hyperbolicity assumptions on $\hat x$, $W_{\rm loc}^u(\hat x)$
projects isomorphically onto a submanifold embedded in a neighborhood of $x_0 \in X$.  Different histories  generally give rise to different local unstable manifolds in $X$.
In our situation, the disks subordinate to $T^-$ play the role of unstable disks (Proposition \ref{lem:separ} and the remark thereafter). 

\begin{defi}
We say that the unstable manifolds of $f$ fully depend  on histories if the
assignment $\hat x \mapsto (x_0, D(\hat x))$ 
is 1-1 on a set of full measure; in other words, if the intermediate projection
$\hat X\cv\check X$ is  a measurable isomorphism.

We say that the unstable manifolds do not depend on histories if $\check\pi$
is a measurable isomorphism but 
$\pi_{0}:\hat X \cv X$ is not (i.e. $f$ is not essentially invertible).
\end{defi}

Przytycki proved in \cite{prz} that for $C^1$-generic Anosov endomorphisms of
tori, the unstable manifolds  depend on histories. 
Mihailescu and Urbanski \cite{murb} have proved the dependence on histories for certain
generic perturbations of saddle sets for holomorphic endomorphisms of $\pd$. 
It is natural to expect that for generic mappings with small topological degree (and not essentially invertible, see Example \ref{ex:henon}),
the unstable manifolds (fully) depend on histories.  However we do not know how to verify this, even on a single example!

\medskip
\begin{center}
 $\diamond$
\end{center}
\medskip

We return to establishing the ergodic properties of $\mu$.  In the next theorem we analyze the conditionals induced by $\hat \mu$ on the
partition by local unstable manifolds in $\hat X$. 

\begin{thm}\label{thm:conditional}
The conditional measures of $\hat{\mu}$ along the unstable partition in
$\hat{X}$ are induced by the current $T^+$. 
\end{thm}

The unstable conditionals are well understood in $\check X$, so we will take advantage of the fact that $\hat X$ is the natural extension of
$\check X$:  if $\hat{p}\in \hat X$ is a generic point, it has a local unstable manifold $W^u_{\rm loc}(\hat{p})$, and $\pi_0$ is a 
homeomorphism from $W^u_{\rm loc}(\hat{p})$ to a disk $D$ through $\pi_0(p)$. The statement of the theorem is that 
$\hat{\mu}_{W^u_{\rm loc}(\hat{p})} = c (\pi_0^{-1})_* (T^+\geom D)$ ($c$ is a normalization constant). Notice here that it is important to consider the natural
extension as a topological, and not only measurable, object. 
To address topological issues in  the natural extension, we use the 
ordinary model of the shift acting on the space of histories, which is naturally a 
compact metric space.
Note also that if the unstable manifolds fully depend on histories, then the result is obvious since $\hat{X}\simeq \check X$. 

\begin{proof}
We view $\hat X$ as the natural extension of $\check X$. So for notational ease 
we let $\pi_0$ denote the natural mapping $\hat X\cv\check X$. Consider also the sequence of projections $(\pi_n)_{n\in \zz}$ with
$\pi_{n+1} = \check f\circ\pi_n$.

By construction, $\check X$ admits a measurable partition into local unstable
disks. 
Consider a flow box $P$ of positive measure in $\check X$, that is, a
sub-lamination of one of the sets $\check{E}_k$ of Theorem 
\ref{thm:tautological}, made up of disks of size $\qq$.  We denote by $p$ the
generic point of $P$, and by 
$\xi$  the natural partition of $P$ by disks, so that $\xi(p)$ is the disk
through $p$. 

From the discussion following Definition \ref{def:dist}, we know that every $\qq$-distinguished 
history of $p\in P$ comes along with a sequence of inverse 
branches of $\xi(p)$. Reducing the size of $P$ if necessary, the set of
$\qq$-distinguished inverse branches  of the disks of $P$ 
forms a set of positive measure (a flow box) in $\hat X$, that we will denote by
$\hat P$ and which is naturally laminated by unstable manifolds. We denote by
$\hat\xi^u$ the partition of $\hat P$ into unstable disks. 
By transitivity of the conditional expectation, the conditionals induced by
$\hat\mu$ on the atoms $\hat\xi$ are also induced by 
$\hat\mu_{\hat P}= \unsur{\hat\mu(\hat P)}\hat\mu\rest{\hat P}$. 
Since we can exhaust $\hat X$ up to a set of arbitrarily small measure with flow
boxes, it will be enough to understand the conditionals on $\hat P$.

\medskip

The partition  $\xi$ induces a  (coarse) partition $\hat{\xi}_0 = \pi_0^{-1}(\xi)$
on $\hat P$, defined by $\hat{\xi}_0(\hat p) = \pi_0^{-1}\xi(\pi_{0}(\hat p))$.
Consider an atom $D$ of the partition $\xi$ on $\check X$ and look at the part
of $\check f^{-n}(D)$ corresponding to the branches belonging to $\hat P$. This
is a union of univalent inverse branches of $\check f^n$, so it inherits a
natural finite partition 
into inverse branches. We can reformulate this as follows: given an atom $C\in
\hat\xi_0$, 
$\pi_{-n}(C)$ admits a partition into finitely many pieces corresponding to the
inverse branches of $\check f^n$ along $\pi_0(C)$. This induces  a refinement of
$\hat\xi_{0}$ that we denote by  $\hat\xi_{-n}$ (``separating inverse branches
of order $n$'', see figure \ref{figure:partition}). It is clear that $\hat\xi_{-n}$ is an increasing sequence of
partitions such that
$\bigvee_{n=0}^\infty \hat\xi_{-n} = \hat \xi^u$, the partition of $\hat P$ into
unstable leaves, up to a set of zero measure.

\begin{figure}[h]
\begin{center}
\psfrag{=0}{$=\hat\xi_0$}
\psfrag{=1}{$=\hat\xi_{-1}$}
\psfrag{=2}{$=\hat\xi_{-2}$}
\psfrag{D}{$D$}
\psfrag{f}{$f$}
\psfrag{hatP}{$\hat P$}
\psfrag{pizero}{$\pi_0$}
\includegraphics[scale=0.5]{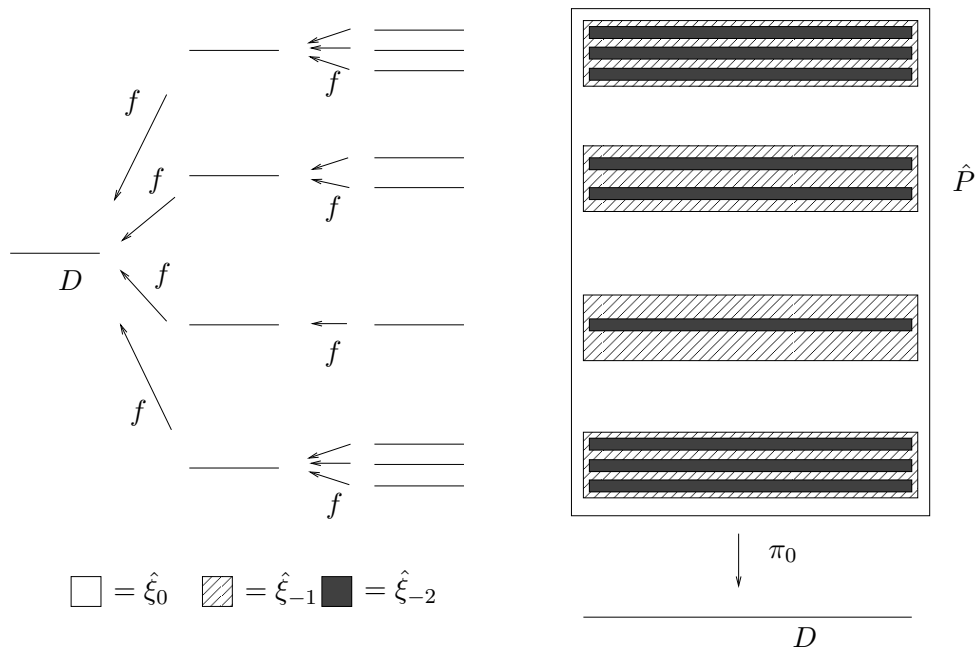}
\caption{Construction of $\hat\xi^u$. Inverse branches of $D$ 
in $\hat P$ are depicted on the left, and the inductive construction of the 
 partition on the right}\label{figure:partition}
\end{center}
\end{figure}

\medskip

By Theorem \ref{thm:tautological}, the conditionals of $\check\mu$ on the atoms
of $\xi$ are induced by $T^+$, i.e.
for a.e. $p$, $\check\mu_{\xi(p)} = c(p) (T^+\geom [\xi(p)])$.
Consider the conditionals induced by $\hat\mu$ on $\hat{\xi}_0$.  Applying the
projection lemma \ref{lem:project}
with $(\tilde Y,\nu)= (\hat X, \hat\mu)$ and the projection $\pi_{0}$, we get
that for almost every atom $C$ of 
$\hat\xi_{0}$, $(\pi_{0})_{*}(\hat\mu_{C})=\check\mu_{\pi_{0}(C)}$.

Consider now the disintegration of $\hat\mu$ relative to the refined partition
$\hat\xi_{-n}$. We will prove that for every $n$, if $C$ is a 
generic atom of $\hat\xi_{-n}$, $(\pi_0)_* \hat\mu_C = \check\mu_{\pi_0(C)}$.
Let us first see why this implies the statement of the theorem.
Recall that $\hat\xi_{-n}$ increases to $\hat\xi^u$. Hence for $\hat\mu$ a.e.
$\hat p$ and every measurable function $\psi$, $\hat \mu_{\hat\xi_{-n}(\hat
p)}(\psi)\cv \hat \mu_{\hat\xi^u(\hat p)}(\psi)$ (this is the Martingale
Convergence Theorem, see \cite{PU}). Now if $\psi$ is of the form
$\varphi\circ\pi_0$, the statement that   $(\pi_0)_* \hat\mu_{\hat\xi_{-n}(\hat
p )} = \check\mu_{\xi(\pi_0(\hat p))}$, implies that  for every $n$, $\hat
\mu_{\hat\xi_{-n}(\hat p)}(\psi) = \check\mu_{\xi(\pi_0(\hat p))} (\varphi)$ is
independent of  $n$. So we get that 
$\hat\mu_{\hat\xi(\hat p )}(\psi) = \check\mu_{\xi(\pi_0(\hat p))} (\varphi)$. In
other words, 
$(\pi_0)_*\hat\mu_{\hat\xi^u(\hat p)} = \check\mu_{\xi(\pi_0(\hat p))}$. But now
$\pi_0$ is a measurable isomorphism $\pi_0:\hat\xi^u(\hat p)\cv\xi(\pi_0(\hat p))$ and
$\check\mu_{\xi(\pi_0(\hat p))}$ is induced by $T^+$, so the proof is finished.

\medskip

It remains to prove our claim that if $C$ is a 
generic atom of $\hat\xi_{-n}$, $(\pi_0)_* \hat\mu_C = \check\mu_{\pi_0(C)}$.
For this, denote by $D$ the disk $D=\pi_{0}(C)$ and notice that
$C=(\pi_{n})^{-1}(D_{-n})$, 
where $D_{-n}$ is the image of $D$ by some inverse branch of  $f^{n}$. By 
Lemma \ref{lem:project} applied to $\pi_{-n}$, we get that 
$(\pi_{-n})_*\hat\mu_C = \check\mu_{D_{-n}}$. Since $\pi_0 = f^n\circ\pi_{-n}$,
we thus obtain that $(\pi_0)_*\hat\mu_C=(f^n)_* \check\mu_{D_{-n}}$. Now we know
that  the conditional  $\check\mu_{D_{-n}}$ is induced by $T^+$, 
and by the invariance property of $T^+$ we have that 
$T^+\wedge [D]= \lambda_1^n (f^n)_*(T^+\wedge [D_{-n}])$. After normalization, we
conclude that $(f^n)_* \check\mu_{D_{-n}} = \check\mu_D$ and the result is
proved.
\end{proof}

We can now compute the entropy, using the Rokhlin formula and the invariant
``Pesin partition'' of \cite{ls}, as in \cite{birat}. 
See \cite{bls} for a nice presentation of the material needed here.
The partition $\xi$ is said to be {\em $f^{-1}$-invariant} if $f^{-1}\xi$ is a
refinement
of $\xi$, i.e. for every $x$, $f^{-1}(\xi(f(x)))\subset \xi(x)$. It is
generating if $\bigvee_{n\geq 0} f^{-n}\xi$ is the partition into points; such
partitions allow the computation of entropy.

\begin{pro}[\cite{ls}]
 There exists a measurable $\hat f^{-1}$-invariant and generating partition of
$\hat X$, whose atoms are open subsets of local unstable manifolds a.s. 
\end{pro}

\begin{proof}
The proposition is stated in the context of  Pesin's theory applied to
diffeomorphisms of manifolds in \cite{ls}, but it is well adapted to our
situation. The exact requirements are listed in \cite[Prop. 3.3]{ls}.
What is needed is a family of local manifolds $V_{\rm loc}(x)$, and a set 
set $\Lambda_\ell$ of measure $\geq 1-\e(\ell)$ such that 
\begin{itemize}
 \item[-] for $x\in \Lambda_\ell$,  the manifolds $V_{\rm loc}(x)$ have
uniformly bounded geometry, and move continuously;
 \item[-] $\hat f^{-n}$ is uniformly exponentially contracting on  $V_{\rm
loc}(x)$, $x\in \Lambda_\ell$.
\end{itemize}

In our situation, we know that $\check X$ can be written as a countable union of
flow boxes as follows: first write $\check X$ as a countable union of flow boxes
$P$. Then write $\pi_0^{-1}(X)$ as a countable union of flow boxes, up to a set
of zero measure as follows:
 consider the increasing sequence of subdivisions
$\qq_i$, subdivide $P$ into smaller flow boxes of size $\qq_i$, and consider the
$\qq_i$-distinguished histories  of the smaller flow boxes. 

Now if $P$ is a flow box of size $\qq$ and $\hat P$ is the set of its 
$\qq$-distinguished histories, 
then $\hat P$ is naturally a flow box of $\hat X$, where the plaques are
unstable manifolds and the transversal is the set of $\qq$-distinguished
histories of a transversal of $P$. Moreover, the dynamics of $\hat f^{-n}$ is
uniformly exponentially contracting along the leaves.
We leave the reader check that \cite[Prop. 3.1]{ls} can now be adapted easily
--see also \cite{qian-zhu} for an 
 adaptation of \cite{ls} to a non-invertible situation, using Pesin's theory. 
\end{proof}

\begin{cor}
 $h(f,\mu) = \log\lambda_1$.
\end{cor}

\begin{proof}
The proof now follows  from a classical argument, which we include for the reader's convenience.
We compute the entropy in the
natural extension. Let $\hat \xi^u$ be the invariant partition constructed
above. Since 
 $\hat \xi^u$ is a generator and $h_{\hat\mu} (\hat f)\leq \log\lambda_1$ is
finite, 
$h_{\hat\mu} (\hat f)$ equals the conditional entropy $h_{\hat\mu} (\hat f,
\hat\xi^u)$. Now the 
conditional entropy may computed by the  Rokhlin formula:
  $$ h_{\hat\mu} (\hat f, \hat \xi^u)=
  -\int \log \mu_{\hat\xi^u(x)}(\hat f^{-1}\hat \xi^u(x) )) d\hat\mu (x) =
 \int \log J^u_{\hat\mu} (x)d\hat\mu(x), $$
 where $J_{\hat\mu}^u(x):= \left({\hat\mu_{\hat\xi^u(x)} \big( \hat
f^{-1}(\hat\xi^u
(\hat f(x)))\big)}\right)^{-1}$ is the unstable Jacobian.

Since the $\hat\xi^u(x)$ are open subsets of unstable manifolds in $\hat X$, by
Theorem \ref{thm:conditional}
the conditionals  $\hat\mu_{\hat\xi^u(x)}$ are induced by $T^+$, that is, with
the usual abuse of notation, 
$$\hat\mu_{\hat\xi^u(x)} =
 \frac{T^+\geom[\hat \xi^u(x)]}{\m\big(T^+\geom[\hat \xi^u(x)]\big)}. $$
 From the invariance relation $f^* T^+ = \lambda_1 T^+$ 
we deduce that
$$T^+\geom[\hat f^{-1}(\hat \xi^u(\hat f(x)))]= \big(T^+\geom [\hat
\xi^u(x)]\big)\big|_{ 
\hat f^{-1}(\hat\xi^u(\hat f(x)))} = \unsur{\lambda_1}\hat f^*(T^+\geom [\hat
\xi^u(\hat f(x))]),$$ hence 
the unstable Jacobian $J^u_{\hat\mu}$ satisfies the multiplicative
cohomological equation 
$$ J^u_{\hat\mu} (x) = \lambda_1 \frac{\rho(x)}{\rho(\hat f(x))} \text{ a.e. , 
where }
\rho(x) = \m\big(T^+\geom [\hat\xi^u(x)]\big).$$
To prove that the integral of $\log J^u_{\hat\mu}$ equals $\log\lambda_1$, we
want to 
use Birkhoff's Ergodic Theorem. The (additive) coboundary
$\log \rho(x) - \log \rho(\hat f(x))$ needn't be integrable, nevertheless by the
invariance of the partition it is bounded from below (see
\cite[Proposition 3.2]{bls} for more details). So  Birkhoff's  Theorem applies
and we conclude that $h_{\hat\mu} (\hat f)=\log\lambda_1$.
\end{proof}

\section{Product structure}\label{sec:product}

We say that a measure has local product structure with respect to stable and unstable manifolds if there is a covering by product subsets of positive measure, in which the vertical (resp. horizontal) fibers are exponentially contracted (resp. expanded) by the dynamics, and the measure is isomorphic to a product measure in each of these subsets.
This property is known to have strong ergodic consequences, like the K and Bernoulli property (see \cite{ow} for more details).

\begin{thm}\label{thm:product}
 The measure $\hat\mu$ has local product structure with respect to local stable
and unstable manifolds in $\hat X$.
\end{thm}

\begin{proof} The proof is in two steps.  We first show that $\check\mu$ has local product structure. This is actually
a statement about the geometry of the currents $T^\pm$.  Next we pass to the natural extension.  What is delicate in this step is 
to analyze the stable conditionals. Observe also  that when unstable manifolds fully depend 
on histories, the second step is automatic.

\medskip

\noindent {\bf Step 1.} We begin by identifying product subsets in $\check X$.  We let $\check \el^u$ denote a ``flow box for $T^-$'', that is, a sublamination  of a flow box in $\check X$.  We may write
$\check \el^u  = \bigcup_{D^u \in\tau^u} D^u$, where $D^u$ are local unstable disks and $\tau^u$ is an (abstract) transversal.  Now suppose that $D_0^s$ is a disk subordinate to  $T^+$ in $X$ and that $D_0^s$ meets some disk 
in $\check\el^u$ transversely in a single point.  By definition of subordinate disks, there is a flow box 
$\el^s = \bigcup_{D^s \in\tau^s} D^s$ of positive mass for $T^+$ made of disks $D^s$ close to $D^s_0$.  Restricting $\check\el^u$ if necessary, we may assume that each disk in $\el^s$ intersects each disk in $\check\el^u$ transversely in a single point. 

We lift leaves $D^s\in\tau^s$ to $\check \el^u$ by setting\footnote{This might look
quite cumbersome at first sight, but be careful that by definition 
the natural lifts of disks to the tautological extension never intersect! Here
we play with the distinction between total and proper transform: we intersect
the proper transform under $\check\pi$ of $D^u$ with the total transform of
$D^s$. We put a check on $\check D^s$ and not on $D^u$ to emphasize the fact
that these two objects do not have the same status.} 
$\check D^s := \check\pi^{-1}(D^s)\rest{\check \el^u}$ and $\check \el^s = \bigcup_{D^s\in \check \tau^s} \check D^s$.
It is clear that $\check \el^u\cap\check \el^s$ is a product set, homeomorphic to $\tau^s\times\tau^u$. For later convenience we identify the abstract transversal $\tau^s$ with a subset of some leaf $D^u\in\tau^u$ and likewise $\tau^u$ with a subset of of some $\check D^s$.
Since $T^+$ and $T^-$ intersect geometrically, there is a countable family of disjoint product sets $\tau^s\times \tau^u$ as above, whose union has full $\check\mu$ measure. 
\medskip

By the analytic continuation theorem for $T^+$ \cite[Theorem 1.1]{structure} we know that $T^+\rest{{\el^s}}$ is uniformly laminar.  Similarly, Theorem \ref{thm:cont} implies that $T^-\rest{\check\pi(\check \el^u)}$ is uniformly woven and can hence be lifted as a uniformly laminar current on the ``abstract'' lamination $\check \el^u$. Abusing notation, we denote this lift by $T^-\rest{\check \el^u}$.  Taking tautological extensions gives rise to a  natural product measure on $\check \el^u\cap\check \el^s$, 
which, abusing notation again, we denote by $T^+\rest{\check{\el^s}}\geom
T^-\rest{\check \el^u}$. To conclude that $\check\mu$ has local product
structure, it remains to prove that this product measure coincides with
$\check\mu\rest{\check \el^u\cap\check \el^s}$.

Recall from Theorem \ref{thm:tautological} that we have constructed $\check\mu$ as the increasing limit of the tautological extensions of $T^+\geom T^-_{\qq_i}$.  The measure $T^+\geom T^-_{\qq_i}$ is in turn defined as an increasing limit of $T^+_{\qq'}\geom T^-_{\qq_i}$. Therefore, the restriction $\check\mu\rest{\check
\el^u\cap\check \el^s}$ is an increasing limit of restrictions of the tautological extensions $(T^+_{\qq'}\geom
T^-_{\qq_i})\rest{\check \el^u\cap\check \el^s}$. By definition of the restriction of 
$T^{+/-}$ to $\el^{s/u}$, all these measures are dominated by
$T^+\rest{\check\el^s}\geom T^-\rest{\check \el^u}$. So we infer that 
$\check\mu\rest{\check \el^u\cap\check \el^s}\leq T^+\rest{\check\el^s}\geom
T^-\rest{\check \el^u}$. On the other hand 
$T^+\rest{\check\el^s}\geom T^-\rest{\check\el^u}$ is a measure supported on
${\check \el^u\cap\check \el^s}$ and dominated by $\check\mu$ so the converse
inequality is obvious.

\medskip

\noindent{\bf Step 1, reinforced.}
To pass to the natural extension we actually need a stronger version 
of the product structure of $\check\mu$, where stable and unstable pieces are
allowed to intersect many times. Consider a flow box $\check\el^u$ as
above, endowed with the abstract uniformly laminar current
$T^-\rest{\check\el^u}$, or equivalently, with an invariant transverse measure
$\mu^-$. Define a {\em measurable transversal} as a measurable set in
$\check\el^u$
intersecting each unstable leaf in $\check\el^u$ along a discrete set. The
transverse measure induces a positive measure $\mu^-_\tau$ on each measurable 
transversal $\tau$, invariant under the equivalence relation defined by the
unstable leaves (see \cite[p.102]{MS}). 
More precisely, if $\tau_1$ and $\tau_2$ are two measurable
transversals and $\phi:\tau_1\cv\tau_2$ is a measurable isomorphism preserving
the leaves, then $\phi_*  \mu^-_{\tau_1} = \mu^-_{\tau_2}$. 

Using this formalism and the geometric intersection of the currents, we have a
finer understanding of the stable conditionals. Consider a set  of positive
measure,  endowed with a partition $\xi^s$, such that each piece $\xi^s$ is
contained a countable union of disks subordinate to $T^+$. We can thus consider
the trace $\hat \xi^s$
of this partition on  $\check\el^u$, as done before. We further assume that each
atom of $\check\xi^s$ intersects the leaves of $\check\el^u$ along a 
discrete set. So it is a measurable transversal to $\check\el^u$. 
The conclusion is that the conditionals of $\check\mu$ (or equivalently, $\check\mu_{\check\el^u}$)
on the pieces $\check\xi^s$ are induced by the transverse measure $\mu^-$, i.e. 
$\check\mu_{\check\xi^s} = c\mu^-_{\check\xi^s}$, with $c$ a normalization
constant.

\medskip

\noindent{\bf Step 2.} As above we view $\hat X$ as the natural extension of
$\check X$, and denote by $(\pi_n)_{n\in \zz}$ the natural projections  $\hat
X\cv\check X$.
Consider a product set $P\simeq \tau^s\times\tau^u$ as defined above. For fine
enough $\qq$, the set of $\qq$-distinguished histories of points in $P$ has
positive $\hat \mu$ measure. Reducing $P$ if necessary, we can assume that $P$
is contained in a flow box of size $\qq$, and denote by 
$\hat P$ the set of $\qq$-distinguished histories with $x_0\in P$.
We will first prove that $\hat P$ is a product
set (see also Figure \ref{figure:tautological})
and next that $\hat\mu\rest{\hat P}$ has product structure.

\medskip

Denote by $\hat\tau^u$ the set of $\qq$-distinguished histories of points in
$\tau^u$.
Recall that  every $\qq$-distinguished history of $p\in P$ comes along
with a sequence of exponentially contracting inverse branches defined on the
disk of size $\qq$ on which $p$ sits. So for every $\hat p\in \hat\tau^u$ 
there exists a lift $\hat D^u(\hat p)$ of $D^u(\pi_0(\hat p))$ such that $\pi_0:
\hat D^u(\hat p)\cv D^u(\pi_0(\hat p))$ is an isomorphism.
Clearly, 
$\hat P \subset \bigcup_{ \hat p\in \hat\tau^u}\hat D^u(\hat p)$. 

On the other hand, for every 
$\hat p \in \hat P$, the local stable manifold of $\hat p$ is the full preimage
under $\pi_0$ of the local stable manifold of $\pi_0(p)$. More precisely, let
$\hat D^s(\hat p)=\pi_0^{-1}  (\check D^s (\pi_0(\hat p)))\cap \hat P$. 
The dynamics are exponentially contracting along the pieces $\hat D^s$. The set
of pieces is parameterized by $\tau^s$, or equivalently by the lift of $\tau^s$ to some unstable disk  
$\hat D^u$ (recall that $\tau^s$ is identified with a subset of some $D^u$).
Since $\pi_0$ is injective on $\hat D^u(\hat p)$ and $\pi_0(\hat D^u(\hat p)
\cap \hat D^s(\hat p))\subset \set{\pi_0(\hat p)}$, 
we infer that $\hat D^u(\hat p) \cap \hat D^s(\hat p) =\set{\hat p}$. We
conclude that $\hat P \simeq \tau^s\times\hat\tau^u$ is a product set.

\medskip

We now show that $\hat\mu$ is a product measure in $\hat P$. Since $\hat P$ is a
product, 
 we have two partitions $\hat D^s$ and $\hat D^u$, with a natural holonomy map
between stable (resp. unstable) pieces. We have to  prove that the
 conditionals induced by $\hat\mu$ on the stable (resp. unstable) pieces are
invariant under holonomy. 

This is easier for unstable conditionals since we know them explicitly. Indeed,
let $\hat D^u_i$, $i=1,2$ be unstable pieces, and 
$\hat h:\hat D^u_1 \cv \hat D^u_2$ be the holonomy map. Define the corresponding
objects  $D^u_i$ and $h$ in $P$ by projecting under $\pi_0$. 
Recall that $\pi_0$ is an isomorphism $\hat D^u_i\cv D^u_i$. We have proved in
Theorem \ref{thm:conditional} that $(\pi_0)_*\hat\mu_{\hat D^u_i} =
\check\mu_{D^u_i}$. In addition, by the product structure of $\check\mu$ in $P$,
we know that $h_* \check\mu_{D^u_1} = \check\mu_{D^u_2}$.  
Since $h\circ\pi_0 = \pi_0\circ\hat h$, we conclude that $\hat h_*\hat\mu_{\hat
D^u_1} = \hat\mu_{\hat D^u_2}$.

\medskip

Due to possible asymmetry between preimages, we cannot give an explicit description of the conditionals of $\hat \mu$ on the stable partition $\hat D^s$. We nevertheless have enough information to  prove holonomy invariance.  That is, if $\hat h:D^s_1\to D^s_2$ is the holonomy map between two stable pieces, we will show that $\hat h_*\hat\mu_{ \hat D^s_1}=\hat\mu_{\hat D^s_2}$.  Consider the measurable set $P_{-n} :=\pi_{-n}(\hat P)$.  We will use the following principle ``the conditionals $\hat\mu_{\hat D^s}$ are completely determined by the conditionals induced by $\check \mu$ on $\check f
^{-n} \check D^s\cap P_{-n}$, and the latter are invariant under holonomy''.

To make the argument more accessible, we first make the  following 
simplifying hypotheses:
\begin{enumerate}[(i)]
\item $\hat P = \pi_{-k}^{-1}(P_{-k})$ for some $k$; 
\item for every unstable piece $\hat D^u$ of the unstable partition of $\hat
P$, 
$\pi_{-n}(\hat D^u)$ is contained in a single flow box of $\check X$.
\end{enumerate}

Restricting our attention to $n\geq k$, we observe that assumption (i) implies $\hat P = \pi_{-n}^{-1}(P_{-n})$. From this and the defining properties of $\hat\mu$ we get that $(\pi_{-n})_*\hat\mu_{\hat P} = \check\mu_{P_{-n}}$.  Let $\check D^{s,-n}$ denote the  pullback partition $\check f^{-n}\check D^s$ of $P_{-n}$. Since $\pi_0 = f^n\circ\pi_{-n}$ we infer that for every $p\in P_{-n}$, we have $\pi_0^{-1} \check D^s(f^n(p))\cap \hat P = \pi_{-n}^{-1}(\check f^{-n}\check D^s(p)\cap P_{-n})$.  Or rather, in our notation, $\hat D^s=\pi_{-n}^{-1} (\check D^{s,-n})$.  Passing to the conditionals, by Lemma \ref{lem:project} we infer that for a.e. $\hat p$ (with $p_{-n} =\pi_{-n}(\hat p)$
\begin{equation}\label{eq:proj}
 (\pi_{-n})_*\hat \mu_{ \hat D^s(\hat p)} =  \check\mu_{\check D^{s,-n}
(p_{-n})}.
\end{equation}

We define the holonomy map (depending on $n$)
between pieces of $\check D^{s,-n}$ naturally as follows: let $\hat D^s_1$ and
$\hat D^s_2$ be two stable pieces and $\hat h$ be the holonomy map between them.
For $p_{-n} = \pi_{-n}(\hat p) \in \check D_1^{s,-n}$, let 
 $h(p_{-n}) = \pi_{-n}(\hat h(\hat p))\in \check D_2^{s,-n}$ (which is
independent of the choice of $\hat p$ 
mapping to $p_{-n}$). By (ii),
the points $p_{-n}$ and $h(p_{-n})$ 
belong to the same flow box of $\check X$, and of course correspond under
holonomy in this flow box.  

Let us prove that the conditionals $\check\mu_{\check D^{s,-n}}$ are invariant under $h$. It is enough to restrict $\check D^{s,-n}$ to some flow box $\check \el^u$. To simplify notation, we continue to denote the restriction by $\check D^{s,-n}$. As the trace in $\el^u$ of a holomorphic disk, $\check D^{s,-n}$ intersects unstable disks along discrete sets (the proof of Proposition \ref{lem:separ} shows that no open subset of $D^{s,-n}$ can be \emph{contained} in an unstable leaf). Hence we are in position to apply the
reinforced version of Step 1:  the conditionals $\check\mu_{\check D^{s,-n}}$ are induced by
the transverse measure associated to  $T^-\rest{\check\el^u}$, and by assumption
(ii), intersection points do not escape $\check\el^u$ by flowing under $h$; hence $h$ defines a measurable isomorphism $\check D_1^{s,-n}\cv \check D_2^{s,-n}$, respecting the leaves. It follows that the conditionals $\check\mu_{\check
D^{s,-n}}$ are invariant under $h$.

\medskip

From \eqref{eq:proj} and this discussion, we deduce for every $n$ and almost
every pair of atoms $\hat D^s_1$, $\hat D^s_2$, that $(\pi_{-n})_*\hat h_* 
\hat\mu_{\hat D^s_1} = (\pi_{-n})_*\hat\mu_{\hat D^s_2}$, where $\hat h$ is the
holonomy map $\hat D^s_1\cv \hat D^s_2$.  It follows that  $\hat h_* \hat\mu_{\hat D^s_1}=\hat\mu_{\hat
D^s_2}$. Indeed we have two measures $\hat h_* \hat\mu_{\hat D^s_1} $ and $\hat\mu_{\hat D^s_2}$ on $\hat D^s_2$,
agreeing on the $\sigma$-algebra $\mathcal{F}_n$ generated by sets of the form $\pi_{-n}^{-1}(A)$.  For every 
$A\subset \hat D^s_2$, we have $A=\bigcap_{n\geq k} \pi_{-n}^{-1}(\pi_{-n}(A))$.  Hence the smallest  $\sigma$-algebra containing all $\mathcal{F}_n$, $n\geq k$, is the Borel $\sigma$-algebra.  The assertion now follows from standard measure theory.

\medskip

What  remains now is to remove the simplifying assumptions. We will 
show that the simplifying assumptions are true ``up to subsets of small measure''.
The details are a bit intricate; we start
with a simple observation. 

\begin{lem}\label{lem:cvproba}
 Let $(A,\nu)$ be a probability space with a measurable partition $\xi$. Assume
that $B_n$ is a sequence of sets with $\nu(B_n)\cv 1$. 
Then $\nu_{\xi(p)}(B_n\cap \xi(p)) = 1-\e(p,n)$, with $\e(\cdot, n)\cv 0$ in
probability as $n\cv\infty$. In particular there is a subsequence such that
$\e(\cdot, n)\cv 0$ a.e.
\end{lem}

Consider the product set $\hat P$ as above, endowed with the partition $\hat
D^s$.  Since we do not necessarily have that $(\pi_{-n})_*\mu_{\hat P} = \mu_{P_n}$, we consider the sequence of sets 
$\pi_{-n}^{-1}(P_n) =\pi_{-n}^{-1} ( \pi_{-n}(P))$ decreasing to $\hat P$.  We let $E_n^s$ 
denote the partition $\pi_0^{-1}(\check D^s)$ of $\pi_{-n}^{-1}(P_n)$. For every $p\in \hat P$, the
sequence $E_n^s(\hat p)$ decreases to $\hat D^s(p)$,
so $\hat\mu_{\pi_{-n}^{-1}(P_n)}\rest{\hat P}$ is proportional to $\mu_{\hat
P}$.  In terms of the conditionals induced on $\pi_0^{-1}(\check D^s)$, this implies that 
\begin{equation} \label{eq:rest}
\mu_{E_n^s(\hat p)}\rest{\hat D^s(p)} = u(\hat p, n) \mu_{\hat D^s(\hat p)}
\end{equation}
where $u(\cdot, n)$ 
is constant on $\hat D^s(p)$ and increases to 1 a.e. With notation as
before, observe that the analogue of \eqref{eq:proj} is now 
\begin{equation} \label{eq:proj2}
 (\pi_{-n})_*\hat \mu_{ E_n^s(\hat p)} =  \check\mu_{\check D^{s,-n} (p_{-n})}.
\end{equation}

We face two problems regarding holonomy. First, $\hat h$ is not defined everywhere on $\pi_{-n}^{-1} P_n$.
Second, the holonomy $h$ is not defined everywhere in $P_{-n}$ because points can escape flow boxes. Let 
$\hat R$ be the set (depending on $n$) of points $\hat p \in \hat P$ such that $\pi_{-n}(\hat D^u(\hat p))$ is contained in a single flow box of $\check X$.  By construction, this is a product set. 
Furthermore, the diameters of the disks $\pi_{-n}(\hat D^u(\hat p))$ are bounded above by $C
n\lambda_1^{-n/2}$, with $C$ uniform in $\hat P$.  Thus as soon as
$\pi_{-n}(\hat p)$ has distance at least $C n\lambda_1^{-n/2}$ from the boundary of a flow
box (in the leafwise direction), we have $ D^u(\hat p) \subset \hat R$. Since $\check\mu$ concentrates no mass on the boundary, we conclude that the relative measure of $\hat R$ in $\hat P$ tends to 1 as $n\cv\infty$.  If $R_{-n} =
\pi_{-n}\hat R$, then the holonomy map $h$ along the leaves of $\check X$ is well
defined on $R_{-n}$ and preserves the conditionals induced by $\mu$ on 
the induced partition $D^{s,-n}\cap R_{-n}$.  The relationships among the sets we have introduced are summed up as follows:
$$
\pi_{-n}^{-1} P_n\supset \hat P \supset \hat R \subset \pi_{-n}^{-1}R_n\subset
\pi_{-n}^{-1} P_n;
$$ 
Moreover, $\hat h$ is well-defined on $\hat P$, $h$ is well-defined on  $R_n$, and $\pi_{-n} \circ
\hat h = h\circ \pi_{-n}$ on $\hat P\cap \pi_{-n}^{-1}R_n$. 
Using \eqref{eq:proj2} we deduce for a.e. $\hat p_1$, $\hat p_2$ that
\begin{equation} \label{eq:h}
h_*\left[\left((\pi_{-n})_*\hat \mu_{ E_n^s(\hat
p_1)}\right)\rest{R_{-n}}\right] = \left( (\pi_{-n})_*\hat \mu_{ E_n^s(\hat
p_2)}\right)\rest{R_{-n}}.
\end{equation}

Restricting measures does not commute with $\pi_{-n}$ but gives at least an 
inequality. Combining it with \eqref{eq:rest}, we find
$$
\left((\pi_{-n})_*\hat \mu_{ E_n^s(\hat p_2)}\right)\rest{R_{-n}}\geq 
(\pi_{-n})_* \left(\hat \mu_{ E_n^s(\hat p_2)}\rest{\hat P\cap
\pi_{-n}^{-1}R_n}\right) = u(\hat p_2, n)
(\pi_{-n})_* \left(\hat \mu_{\hat D^s(\hat
p_2)}\rest{\pi_{-n}^{-1}R_n}\right).
$$
The left side of the inequality is a measure of at most unit mass, whereas by Lemma \ref{lem:cvproba} the mass of the right side is of the form $1-\e(\hat p_2,n)$ (here $\e(\cdot, n)$ denotes a sequence of functions converging in probability to zero, possibly changing from line to line). So the right and left sides differ by a measure of mass of at most $\e(\hat p_2,n)$. By Lemma \ref{lem:cvproba} again, $\hat \mu_{\hat D^s(\hat
p_2)}\rest{\pi_{-n}^{-1}R_n}$ is close to $\hat \mu_{\hat D^s(\hat p_2)}$ in mass.  Therefore, finally we see that
\begin{equation} \label{eq:mass1}
\m \left(\left((\pi_{-n})_*\hat \mu_{ E_n^s(\hat p_2)}\right)\rest{R_{-n}} -
(\pi_{-n})_*\hat \mu_{\hat D^s(\hat p_2)} \right) = \e(\hat p_2,n).
\end{equation}
Since  
$\pi_{-n} \circ \hat h = h\circ \pi_{-n}$ on $\hat P\cap \pi_{-n}^{-1}R_n$ the
left side of \eqref{eq:h} similarly satisfies 
$$
h_*\left[\left((\pi_{-n})_*\hat \mu_{ E_n^s(\hat
p_1)}\right)\rest{R_{-n}}\right] \geq u(\hat p_1, n)
(\pi_{-n})_* \hat h_* \left(\hat \mu_{\hat D^s(\hat
p_1)}\rest{\pi_{-n}^{-1}R_n}\right).
$$
Applying the same reasoning on masses yields
\begin{equation} \label{eq:mass2}
\m \left(h_*\left[\left((\pi_{-n})_*\hat \mu_{ E_n^s(\hat
p_1)}\right)\rest{R_{-n}}\right] - (\pi_{-n})_*\hat h_*
\hat \mu_{\hat D^s(\hat p_1)} \right) = \e(\hat p_1,n).
\end{equation}
From \eqref{eq:mass1} and \eqref{eq:mass2} we conclude that 
$$
\m\left( (\pi_{-n})_*\hat h_* \hat \mu_{\hat D^s(\hat p_1)}  - 
(\pi_{-n})_*\hat \mu_{\hat D^s(\hat p_2)}\right) = 
\e(\hat p_1,n)+\e(\hat p_2,n).
$$ 
Consider as before the $\sigma$-algebra $\mathcal{F}_n$
generated by  $\pi_n$.  Given 
$A$ in $\mathcal{F}_n$, we get that $\hat h_* \hat \mu_{\hat D^s(\hat p_2)}(A) $
and $(\pi_{-n})_*\hat \mu_{\hat D^s(\hat p_2)}(A)$ differ by at most 
$\e(\hat p_1,n)+\e(\hat p_2,n)$. But since $A\in \mathcal{F}_m$ for all $m\geq
n$, we may pass to a subsequence and arrange for a.e. $\hat p_1$, $\hat p_2$ that the difference tends to zero.
Hence $\hat h_* \hat \mu_{\hat D^s(\hat p_1)}=\hat \mu_{\hat D^s(\hat p_2)}$, which is what we wanted to prove.
\end{proof}

\section{Saddle points}\label{sec:saddle}

In this section we prove Theorem \ref{thm:saddle}. The proof is more classical and follows \cite{bls2} closely. 

\medskip

\noindent {\bf Step 1: Pesin theory.}
As explained in the introduction, the proof relies on Pesin's theory of non-uniformly hyperbolic dynamical systems.
The applicability of this theory in our context only requires the assumption that
$p\mapsto \log d(p, I^+\cup C_f)\in L^1(\mu)$, as is neatly shown in \cite{dt-lyap}. 

By using the Osedelets theorem or the foregoing study, we have a set of full $\hat\mu$ measure 
provided with an invariant splitting of the tangent space $T_pX = E^u(\hat p)\oplus E^s(p)$ into an (exponentially) expanding direction and a contracting direction which depend measurably on $\hat p$. 
As indicated by the notation, $E^s(p)$  depends only on $p=\pi_0(\hat p)$. 
By Pesin Theory, there exists an invariant set $\hat{\mathcal{R}}\subset\hat X$ of full $\hat\mu$ measure such that
for each $\hat p\in \hat{\mathcal{R}}$, there exists {\em Lyapunov chart} $L(\hat p)$, which is a topological bidisk in $X$ (in the terminology of \cite{bls2}) centered at $p=\pi_0(\hat p)$. The bidisk can be chosen to be the image, under the Riemannian exponential map, of an affine bidisk with axes $E^u(\hat p)$, $E^s(p)$ and measurably varying size $r(\hat p)$.  We can further assume that $L(\hat p)$ does not intersect $I^+\cup C_f$ or $f(I^+\cup C_f)$ and that $f\rest{L(\hat p)}$ is injective.

The Lyapunov charts have the fundamental property that $f:L(\hat p)\cv L(\hat f\hat p)$ defines a H{\'e}non-like map of degree 1 (we use the terminology of \cite{hl}). That is, the cut-off image of a graph over the horizontal (i.e. unstable) direction, is a graph. This property is referred to in \cite{bls2} as the ``$u$-overflowing property'' of Lyapunov charts.  The branches of $f^{-1}$ have the overflowing property in the vertical (i.e. stable) direction.

We can also consider the sets 
$$
L_n^s(\hat p) := \set{y\in L(\hat p), \ \forall 1\leq j\leq n, \ f^j(y)\in L(\hat f^j\hat p)} 
\text{ and }
L_n^u(\hat p) := f^n L_n^s(\hat f^{-n} \hat p),
$$ 
which converge exponentially fast to the local stable manifold $W^s_{\rm loc}(p)$ and unstable manifold 
$W^u_{\rm loc}(\hat p)$, respectively.  Note that depending on the context, we will sometimes regard local stable and unstable manifolds as subsets of $X$ and sometimes as subsets of $\hat X$.

\medskip

We have shown in the proof of Theorem \ref{thm:complete}{\em i.} 
that disks subordinate to $T^+$ are $\mu$-a.e. exponentially contracted by $f$, while disks subordinate to $T^-$ are $\hat\mu$-a.e. exponentially contracted by distinguished preimages of $f$. On the other hand, the Pesin stable and unstable manifolds are unique.  Therefore, Pesin stable and unstable manifolds coincide a.e. with disks subordinate to $T^+$ and $T^-$.

\medskip

If $\hat p$ and $\hat q$ are sufficiently close in $\hat X$, then $W^s_{\rm loc}(p)\cap W^u_{\rm loc}(\hat p)$ is a single point classically denoted by $[\hat p,\hat q]$.  A subset is said to have product structure if it is closed under the operation $[\cdot, \cdot]$.  A {\em Pesin box} is a compact, positive measure subset of $\hat X$ with product structure and a positive lower bound on the size of the associated Lyapunov charts.

\medskip

\noindent {\bf Step 2: constructing saddle points.} The basic step in the argument is the following: if $g$ is a H\'enon-like map of degree 1 in some topological bidisk $B$, then $\bigcap_{k\in \zz} g^{k}(B)$ is a single saddle fixed point $q$ of $g$. Similarly $\bigcap_{k\geq 0} g^{k}(B) = W^u_B(p)$ and $\bigcap_{k\leq 0} g^{k}(B) = W^s_B(p)$.
Here we employ truncated iteration, in which points are omitted once they leave $B$.

\medskip

Fix $\e>0$. There exists a compact set $\hat{\mathcal{R}}_\e$ with $\hat\mu\big(\hat{\mathcal{R}}_\e\big)>1-\e$, and where all constants appearing above, as well as the stable and unstable directions and manifolds, vary continuously.  As argued in \cite[Lemma 1]{bls2} (stated in the context of polynomial automorphisms, but the proof extends without change to our situation), 
given $\eta>0$, there exist finitely many Pesin boxes, each of diameter smaller than $\eta$, covering $\hat{\mathcal{R}}_\e$. 
 Let $\hat P$ be one of these Pesin boxes. Since the stable and unstable directions of points belonging to $\hat P$ are almost parallel, if $\eta$ is sufficiently small, there exists a ``common Lyapunov chart'' $B$, which is a topological bidisk such that   $$\pi_0(\hat P) \subset B\subset \bigcap_{\hat p\in \hat P} L(\hat p).$$

Now if $n$ is large enough and $\hat p\in \hat P\cap \hat f^{-n}(\hat P)$ (of course there are infinitely many such $n$), $f^n$ induces a H\'enon-like map of degree 1 in $B$ that sends $p = \pi_0(\hat p)$ to $f^n(p)$.  We infer that $f^n$ has a saddle fixed point $q$ exponentially (in $n$) close to $\pi_0([ \hat p, \hat f^n(\hat p)]) \subset \pi_0(\hat P)$.

Now let $\hat q$ be the unique periodic point in $\hat X$ projecting to $q$.  We claim that $\hat q$ is close to $\hat P$. By construction, for $k\in \nn$ and $0\leq i\leq n-1$, $f^{i+kn}(q)\in L(\hat f^i \hat p)$. Since $\hat q$ is periodic we infer that for $k\in \zz$, $\pi_0 \hat f^{i+kn}(\hat q)\in L(\hat f^i \hat p)$. In particular if $i_0$ is a fixed integer we can arrange so that for $\abs{i}\leq i_0$,  $\pi_0 \hat f^{-i}(\hat q)$ is close to 
$\pi_0\hat f^{-i}(\hat P)$. Thus $\hat q$ is close to $\hat P$ relative to  the product metric on $\hat X$, as claimed.

At this stage we know that saddle periodic points accumulate everywhere on $\supp(\mu)$ and $\supp(\hat \mu)$.

\medskip

\noindent{\bf Step 3: equidistribution.} Given $\hat p\in \hat P\cap \hat f^{-n}(\hat P)$ as in Step 2, 
let $B_n^s(p)$ be the connected component of $B\cap f^{-n}B$ containing $p$. Let us first show that $B_n^s(p)\subset L_n^s(\hat p)$. For this, suppose that there exists $p'\in B_n^s(p)$ and a smallest integer $i\geq 0$ such that $f^i(p')\notin L(\hat f^i(\hat p))$. Let $\gamma$ be a path joining $p$ and $p'$; $f^i(\gamma)$ is not contained in $L(\hat f^i(\hat p))$, so it must intersect the vertical boundary of that polydisk. Hence by the H\'enon-like property, for all subsequent iterates $j\geq i$, $f^j(\gamma)$ intersects
the vertical boundary of $L(\hat f^j(\hat p))$, which contradicts the fact that $f^n(B_n^s(p))\subset B$.

Since $B_n^s(p)\subset L_n^s(\hat p)$, the behavior of $f^n$ on $B_n^s(p)$ is that of $n$ successive iterates of a H\'enon-like map of degree 1. In particular we infer that  $B_n^s(p)$ is a vertical sub-bidisk of $B$, while $f^n(B_n^s(p))$ is horizontal. Also, exactly as in 
\cite[Lemma 3]{bls2} we get that 
\begin{equation}\label{eq:Bns}
 W^s_{loc}(p)\cap B\subset B_n^s(p) \text{ and }W^u_{loc}(\hat p)\cap B_n^s(p)\subset \left(f^n\rest{L_n^s(\hat p)}\right)^{-1}
 \left(W^u_{loc}(\hat f^n \hat p)\cap B\right).
\end{equation}

The sets  $\pi_0^{-1}(B_n^s)$ induce a  partition of $\hat P\cap \hat f^{-n}(\hat P)$. If $T$ is an atom of this partition, it is clear that the construction of Step 2 associates  to all points  $\hat p\in T$ the same saddle point $q = q(T)$, and that 
the mapping $T\mapsto q(T)$ is injective. From \eqref{eq:Bns} and the fact that the local stable manifold of $\hat p$ in $\hat X$ is the 
full preimage under $\pi_0$ of $W^s_{\rm loc}(p)\subset X$, we get that $T$ has product structure.
That is, if $\hat p_1,\hat p_2\in T$, then $[\hat p_1,\hat p_2]\in T$.

\medskip

Now since  $\hat\mu$ has product structure and its unstable conditionals are induced by the current $T^+$ (Theorems \ref{thm:conditional} and \ref{thm:product}), we can reproduce \cite[Lemma 5]{bls2} and conclude that for any atom $T$ of the partition, 
$\hat\mu(T)\leq \lambda_1^{-n}\hat\mu(\hat P)$.

Let $\mathrm{SFix}_n$ be the set of saddle periodic points  of period (dividing) $n$ in $\hat X$, and 
$\hat\nu_n = \lambda_1^{-n}\sum_{q\in \mathrm{SFix}_n} \delta_q$. 
If $\hat P_\delta$ denotes the $\delta$-neighborhood  of $\hat P$, we obtain that for large $n$, 
\begin{equation}\label{eq:saddle_estimate}
 \lambda_1^{-n}\hat\mu(\hat P) \# \set{\mathrm{SFix}_n\cap \hat P_\delta} = \hat\mu(\hat P) \hat\nu_n(\hat P_\delta)
\geq \sum \hat\mu(T) = \hat\mu(\hat P\cap \hat f^{-n} \hat P) \longrightarrow \hat\mu(\hat P)^2,
\end{equation}
whence $\liminf \hat\nu_n(\hat P_\delta)\geq \hat\mu(\hat P)$.  If $\hat\nu$ is any cluster point of the sequence $\hat\nu_n$, we conclude that for any $\delta>0$, $\hat\nu(\hat P_\delta)\geq \hat\mu(\hat P)$, thus $\hat\nu(\hat P)\geq \hat\mu(\hat P)$. Since any open subset in $\hat X$ can be covered up to a set of small $\hat\mu$ measure by disjoint Pesin boxes, we conclude that $\hat\nu\geq\hat\mu$

Now assume (see Step 5 below) that we have an estimate $\#\mathrm{SFix}_n\leq \lambda_1^n+o(\lambda_1^n)$.  
Then  $\limsup \nu_n(X) \leq 1 \leq \mu(X)$.  From this and the previous paragraph we conclude that $\mu = \lim \nu_n$. 

In the (unexpected?) event that the estimate $\#\mathrm{SFix}_n\leq \lambda_1^n+o(\lambda_1^n)$ fails, then we can certainly replace $\mathrm{SFix}_n$ with a subset $\mathcal{P}_n\subset \mathrm{SFix}_n$ of size $\#\mathcal{P}_n\sim\lambda_1^n$ for which \eqref{eq:saddle_estimate} remains valid relative to a countable collection of Pesin boxes sufficient to provide disjoint ``near covers'' of any open set.  Hence the measures $\nu_n$ defined using $\mathcal{P}_n$ instead of $\mathrm{SFix}_n$ will again converge to $\mu$.

\medskip

\noindent{\bf Step 4: saddle points in $\supp(\mu)$.} We now prove that the saddle points constructed in Step 2 lie inside $\supp(\mu)$. Here the argument is identical to \cite{birat}, we sketch it for completeness (see also \cite{bls}). 

Given a Pesin box $\hat P$ as above, let $\el^s\subset X$  be the local stable lamination of $\pi_0(\hat P)$, that is, the union of local stable manifolds. Likewise we can define the local unstable web $\el^u = \bigcup_{\hat p\in \hat P}W^u_{\rm loc}(\hat p)$. Let $S^+=T^+\rest{\el^s}$ and $S^-= T^-\rest{\el^u}$. These currents are uniformly geometric and dominated by $T^+$ and $T^-$, respectively. Furthermore, from our understanding of $\mu$ we know that $S^+>0$ and $S^->0$. With notation as in Step 2, let now $g$ be the H\'enon-like map in $B$ induced by $f^n$. 
Therefore, $\lambda_1^{-nk} (g^k)^*S^+$ is a non-trivial, uniformly laminar current dominated by $T^+$.  As $n\to\infty$ its support converges in the Hausdorff sense to the local stable manifold of $q$, where $q = \bigcap_{n\in \zz} g^k(B)$ is the saddle point that we have just constructed.  Similarly, we have corresponding currents $0<\unsur{\lambda_1^{kn}}(g^k)_*S^-\leq T^-$ with supports converging to the local unstable manifold of
(the periodic history of) $\hat q$.  Hence we obtain a sequence of measures
$$
0<\mu_k = \unsur{\lambda_1^{kn}}(g^k )^*S^+ \wedge \unsur{\lambda_1^{kn}}(g^k )_*S^- \leq \mu
$$
with supports converging to $q$.  We conclude that $q\in \supp(\mu)$.

\medskip

\noindent{\bf Step 5: counting periodic points.} The second statement in Theorem \ref{thm:saddle} requires a bound of the form 
$\# Per_n\leq \lambda_1^n+o(\lambda_1^n)$ 
on the number of isolated periodic points. When $f$ has no curves of periodic points, this is classical.
The Lefschetz Fixed Point Formula \cite[p.314]{fulton} asserts that 
$$\set{\Delta}\cdot\set{\mathrm{Graph}(f^k)} = \sum_{0\leq p,q\leq 2} (-1)^{p+q}
\mathrm{Trace}\left((f^k)^*\rest{H^{p,q}(X,\re)}\right),$$ where $\Delta$ is the diagonal in $X\times X$. If $f^k$ has no curve of fixed points, 
the intersection product is the sum of the multiplicities of periodic points plus a (nonnegative) term coming from the indeterminacy set 
(see e.g. \cite[\S 8]{df}). 
All components of $\Delta\cap\mathrm{Graph}(f^k)$ have dimension 0, hence give positive contribution to the intersection product. In particular the 
left hand side of this inequality dominates the number of fixed points of $f^k$. 

On the other hand, the dominating term on the right hand side is given by the trace of the action on $H^{1,1}$, which is  $\lambda_1^n+o(\lambda_1^n)$.
Indeed, observe first that by the small topological degree assumption, 
the action on $H^2$ predominates. Now, when $X$ is rational, $\mathrm{dim} H^{0,2} = \mathrm{dim} H^{2,0} = 0$ so we are done. 
When $X$ is irrational it can be checked directly that $\norm{(f^k)^*\rest{H^{0,2}}}\lesssim \lambda_2^{n/2}$ (a general argument for this is given in \cite[Proposition 5.8]{dinh}), and we are also done in this case.

\medskip

When $f$ admits curves of periodic points and $X=\pd$ or $X=\pu\times \pu$, 
the result follows from a slightly more  sophisticated argument 
from Intersection Theory \cite[\S 12.2]{fulton}. We thank Charles Favre for indicating this to us.
Indeed in this case, the intersection in the Lefschetz fixed point formula takes place in $X\times X = \pd\times \pd$ or $(\pu\times\pu)^2$. This manifold has the property that its tangent bundle is generated by its sections (see \cite[Example 12.2.1]{fulton}), in which case the contribution of every irreducible component of $\Delta\cap\mathrm{Graph}(f^k)$ to the intersection product $\set{\Delta}\cdot\set{\mathrm{Graph}(f^k)}$ is nonnegative 
\cite[Corollary 12.2]{fulton} (see also \cite[Example 16.2.2]{fulton}). As before we conclude that the number of isolated fixed points  of
 $f^n$ is controlled by $\lambda_1^n+o(\lambda_1^n)$. 
 
Observe that if $X$ is any rational surface, by using a birational conjugacy between $f$ and a (possibly non algebraically stable) map on $\pd$, this argument shows that 
the number of isolated fixed points of $f^n$ is controlled by $C\lambda_1^n+o(\lambda_1^n)$ for some $C$.

\begin{rqe}
 Under the same assumptions as in Theorem \ref{thm:saddle} we can also adapt 
 \cite[Theorem 2]{bls2} and obtain that the Lyapunov exponents of $\mu$
 can be evaluated by averaging on saddle orbits, which is an important fact in bifurcation theory.
\end{rqe}


\begin{thebibliography}{[ABC]}
 \bibitem [BeDi]{bedford-diller} Bedford, Eric; Diller, Jeffrey. 
 {\em Energy and invariant measures for birational surface 
  maps.}   Duke Math. J.  {128}  (2005),  no. 2, 331--368
\bibitem[BS]{bs3} Bedford, Eric; Smillie, John. {\em Polynomial
diffeomorphisms of $\cc^ 2$. III. Ergodicity, exponents and entropy of
the equilibrium measure.} Math. Ann. 294 (1992),  395-420.
\bibitem[BLS1]{bls} Bedford, Eric;  Lyubich, Mikhail;  Smillie, John. {\em
Polynomial
diffeomorphisms of $\cc^ 2$. IV. The measure of maximal entropy and
laminar currents.} Invent. Math.  112  (1993), 77-125.
\bibitem[BLS2]{bls2}
Bedford, Eric;  Lyubich, Mikhail;  Smillie, John.
{\em Distribution of periodic points of polynomial diffeomorphisms of
  $\cd$.} Invent. Math.  114 (1993), 277-288.
 \bibitem[Br]{briend} Briend, Jean-Yves. {\em Propri\'et\'e de Bernoulli pour les
extensions naturelles des endomorphismes de $\mathbf{CP}^k$.}
Ergodic Theory Dynamical Systems 21 (2001), 1001-1007. 
\bibitem[BrDu]{briend-duval} Briend, Jean-Yves; Duval, Julien. {\em Deux
caract\'erisations de la mesure 
d'\'equilibre des endomorphismes de $\mathbf{CP}^k$.}
Publ. Math. Inst. Hautes {\'E}tudes Sci.  No. 93  (2001), 145-159. 
\bibitem[Ca]{cantat} Cantat, Serge. {\em Dynamique des automorphismes des
  surfaces K3.}  Acta Math. 187 (2001) 1-57.  
\bibitem[dT1]{dt-saddle} De Th\'elin, Henry. {\em Sur la construction de mesures
selles.} 
Ann. Inst. Fourier, 56 (2006), 337-372.
\bibitem[dT2]{dt-lyap} De Th\'elin, Henry. {\em Sur les exposants de Lyapounov
des applications m\'eromorphes.}  Invent. Math. 172 (2008), 89-116.
\bibitem[DV]{dv} De Th\'elin, Henry; Vigny, Gabriel. {\em Entropy of meromorphic maps and  dynamics 
of birational maps.} Preprint (2008). M\'emoires de la SMF, to appear.
\bibitem[DDG1]{part1} Diller, Jeffrey; Dujardin, Romain; Guedj, Vincent. 
{\em Dynamics of meromorphic maps with small topological degree I: from
cohomology to currents.}  Indiana Univ. Math. J., to appear.
\bibitem[DDG2]{part2} Diller, Jeffrey; Dujardin, Romain; Guedj, Vincent. 
{\em Dynamics of meromorphic maps with small topological degree II: energy and
invariant measure.} Comment. Math. Helvet., to appear.
 \bibitem [DF]{df} Diller, Jeffrey; Favre, Charles. 
 {\em  Dynamics of bimeromorphic maps of surfaces.}   
   Amer. J. Math.  {123}  (2001),  no. 6, 1135--1169.
\bibitem[Di]{dinh} Dinh, Tien Cuong. {\em Suites d'applications m\'eromorphes
multivalu\'ees et courants laminaires.}
 J. Geom. Anal.  15  (2005),  207--227. 
\bibitem[DS]{ds}  Dinh, Tien Cuong; Sibony, Nessim.
{\sl Dynamique des applications d'allure polynomiale.  }
J. Math. Pures Appl. (9) 82 (2003), 367-423.
\bibitem[Du1]{hl} Dujardin, Romain. {\em H\'enon-like mappings in $\cd$.}
Amer. J. Math. 126 (2004), 439-472.
\bibitem[Du2]{isect} Dujardin, Romain. {\em Sur l'intersection des courants
laminaires.}  Pub. Mat. 48 (2004), 107-125.   
\bibitem[Du3]{structure} Dujardin, Romain. {\em Structure properties of laminar
currents.} J. Geom. Anal. 15  (2005), 25-47.
\bibitem[Du4]{birat} Dujardin, Romain. {\em Laminar currents and birational
dynamics.}
Duke Math. J. 131 (2006), 219-247. 
\bibitem[Duv]{duval} Duval, Julien. {\em Singularit\'es des courants d'Ahlfors.}
  Ann. Sci. \'Ecole Norm. Sup. (4)  39  (2006),  527--533.
 \bibitem [FJ]{fj} Favre, Charles; Jonsson, Mattias. 
 {\em Dynamical compactifications of $\cd$.}  Preprint (2007).  Annals of Math., to appear.
 \bibitem[FS]{fs94}Fornaess, John Erik; Sibony, Nessim 
{\em Complex dynamics in higher dimension. {II}}, Modern methods in complex analysis, Ann. of Math. Studies
 135--182. Princeton Univ. Press.
\bibitem[FLM]{flm} Freire, Alexandre; Lopes, Artur; Ma\~n\'e, Ricardo
{\em An invariant measure for rational maps.}
Bol. Soc. Brasil. Mat. 14 (1983),  45--62. 
\bibitem[Fu]{fulton} Fulton, William, {\em Intersection theory.} Second edition. Ergebnisse der Mathematik und ihrer Grenzgebiete. 3. Folge. A Series of Modern Surveys in Mathematics. Springer-Verlag, Berlin, 1998.
\bibitem[Gh]{ghys} Ghys, Etienne. {\em Laminations par surfaces de
Riemann.} Dynamique et g{\'e}om{\'e}trie complexes (Lyon, 1997), 
  Panoramas et Synth{\`e}ses, 8, 1999. 
\bibitem[Gr]{gromov} Gromov, Mikhail. {\em On the entropy of holomorphic maps.}  Enseign. Math. (2)  49  (2003),  217--235.
\bibitem [G1]{g} Guedj, Vincent. {\em Dynamics of polynomial mappings of $\mathbf{C}^2$.}   
  Amer. J. Math.  {\bf 124}  (2002),  no. 1, 75--106. 
\bibitem [G2]{g05} Guedj, Vincent. 
{\em Ergodic properties of rational mappings with large
   topological degree.}  Ann. of Math. {161} (2005).
\bibitem [G3]{g05b} Guedj, Vincent. 
{\em  Entropie topologique des applications m\'eromorphes.}  
 Ergodic Theory Dynam. Systems  {25}  (2005),  no. 6, 1847--1855. 
\bibitem[Led]{ledrappier-srb} Ledrappier, Fran{\c c}ois.
 {\em Some properties of absolutely continuous
  invariant measures on an interval},  Ergodic Theory Dynam. Systems 1 (1981), 77--93.
\bibitem[LS]{ls} Ledrappier, Fran{\c c}ois; Strelcyn, Jean-Marie. 
{\em A proof of the estimation from below in Pesin's entropy formula.} 
Ergodic Theory Dynam. Systems  2  (1982), 203--219.  
\bibitem[Lel]{lelong} Lelong, Pierre. {\em Propri\'et\'es m\'etriques des
vari\'et\'es analytiques complexes
 d\'efinies par une \'equation.} Ann. Sci. \'Ecole Norm. Sup. (3)  67,  (1950).
393--419.
\bibitem[Ly]{lyubich} Lyubich, Mikhail
{\em  Entropy properties of rational endomorphisms of the Riemann  sphere.} 
  Ergodic Theory Dynam. Systems  3  (1983),  351-385.
\bibitem[MU]{murb} Mihailescu, Eugen; Urba\'nski, Mariusz. {\em Holomorphic maps
for which the unstable manifolds depend on histories.}  Discrete Contin. Dyn.
Syst.  9  (2003),  no. 2, 443--450. 
\bibitem[MS]{MS} Moore, Calvin C.; Schochet, Claude. {\em Global analysis on
foliated spaces.} Mathematical Sciences Research Institute Publications, 9.
Springer-Verlag, New York, 1988. 
\bibitem[OW]{ow} Ornstein, Donald; Weiss, Benjamin. {\em On the Bernoulli nature
of systems with some hyperbolic structure.}
  Ergodic Theory Dynam. Systems  18  (1998),   441--456.
\bibitem[Pr]{prz} Przytycki, Feliks. {\em Anosov endomorphisms.} Studia Math. 58
(1976), 249--285.
\bibitem[PU]{PU} Przytycki, Feliks;  Urba\'nski, Mariusz. {\em Conformal fractals, Ergodic Theory methods.} Book to appear,
 available online at 
{\tt http://www.math.unt.edu/$\sim$urbanski/book1.html}
\bibitem[QZ]{qian-zhu} Qian, Min; Zhu, Shu.
{\em SRB measures and Pesin's entropy formula for endomorphisms.} Trans. Amer.
Math. Soc. 354 (2002), 1453--1471.
\bibitem[Ro]{rokhlin} Rohlin, V. A. {\em Lectures on the entropy theory of
transformations with invariant measure.}  Uspehi Mat. Nauk  22  1967 no. 5
(137), 3--56. Transl. in Russian Math. Surveys 22 (1967), no. 5, 1--52.
\bibitem[Y]{y}  Yomdin, Yuri. {\em Volume growth and entropy.} 
Israel J. Math. 57 (1987), 285-300.

\end{thebibliography}
\end{document}